\documentclass[letterpaper,11pt]{article}
\usepackage{color}
\usepackage{verbatim}
\usepackage[small]{titlesec}
\usepackage{theorem,amsmath,amssymb,amscd}
\usepackage[all]{xy}
\usepackage{booktabs}
\usepackage[hyperfootnotes=false, colorlinks, linkcolor={blue}, citecolor={magenta}, filecolor={blue}, urlcolor={blue}]{hyperref}

\theoremstyle{change}
\allowdisplaybreaks
\newcommand{\disc}{{\rm disc}}
\newcommand{\A}{{\mathbb A}}
\newcommand{\Q}{{\mathbb Q}}
\newcommand{\Z}{{\mathbb Z}}
\newcommand{\R}{{\mathbb R}}
\newcommand{\C}{{\mathbb C}}
\newcommand{\HH}{{\mathbb H}}

\newcommand{\p}{\mathfrak p}
\newcommand{\OF}{{\mathfrak o}}
\newcommand{\GL}{{\rm GL}}
\newcommand{\PGL}{{\rm PGL}}
\newcommand{\SL}{{\rm SL}}
\newcommand{\SO}{{\rm SO}}
\newcommand{\Symp}{{\rm Sp}}
\newcommand{\GSp}{{\rm GSp}}
\newcommand{\PGSp}{{\rm PGSp}}
\newcommand{\meta}{\widetilde{\mathrm{SL}}_2}

\newcommand{\SSp}{{\rm Sp}}
\DeclareMathOperator{\Cl}{Cl}
\newcommand\blfootnote[1]{%
  \begingroup
  \renewcommand\thefootnote{}\footnote{#1}%
  \addtocounter{footnote}{-1}%
  \endgroup
}
\newcommand{\mat}[4]{{\setlength{\arraycolsep}{0.5mm}\left[
\begin{array}{cc}#1&#2\\#3&#4\end{array}\right]}}
\newcommand{\qed}{\hspace*{\fill}\rule{1ex}{1ex}}
\newcommand{\forget}[1]{}

\def\qdots{\mathinner{\mkern1mu\raise0pt\vbox{\kern7pt\hbox{.}}\mkern2mu
\raise3.4pt\hbox{.}\mkern2mu\raise7pt\hbox{.}\mkern1mu}}

\newenvironment{proof}{\vspace{1ex}\noindent{\it Proof.}\hspace{0.1em}}
	{\hfill\qed\vspace{2ex}}

\newtheorem{lemma}{Lemma.}[section]
\newtheorem{theorem}[lemma]{Theorem.}
\newtheorem{corollary}[lemma]{Corollary.}
\newtheorem{proposition}[lemma]{Proposition.}

\begin{document}

%\fbox{
%\textcolor{blue}{Text in blue: Keep only for short version.}
%\textcolor{red}{Text in red: Keep only for long version.}
%}

\thispagestyle{empty}
\begin{center}
 {\bf\Large Local and global Maass relations (expanded version)}

 \vspace{3ex}
 {Ameya Pitale, Abhishek Saha and Ralf Schmidt}\blfootnote{The first and third authors are supported by NSF grant DMS-1100541.}

 \vspace{3ex}
 \begin{minipage}{75ex}
  \small {\sc Abstract.} We characterize the irreducible, admissible, spherical representations of $\GSp_4(F)$ (where $F$ is a $p$-adic field) that occur in certain CAP representations in terms of relations satisfied by their spherical vector in a special Bessel model. These local relations are analogous to the Maass relations satisfied by the Fourier coefficients of Siegel modular forms of degree $2$ in the image of the Saito-Kurokawa lifting. We show how the classical Maass relations can be deduced from the local relations in a representation theoretic way, without recourse to the construction of Saito-Kurokawa lifts in terms of Fourier coefficients of half-integral weight modular forms or Jacobi forms. As an additional application of our methods, we give a new characterization of Saito-Kurokawa lifts involving a certain average of Fourier coefficients.
 \end{minipage}
 \end{center}

%\fbox{Omit table of contents in short version, keep for long version.}

\tableofcontents

\section{Introduction}
Let $F$ be a holomorphic Siegel modular form of degree $2$ and weight $k$ with respect to the full Siegel modular group $\SSp_4(\Z)$. Then $F$ has a Fourier expansion of the form
\begin{equation}\label{Four-exp-eq}
 F(Z) = \sum\limits_S a(S) e^{2 \pi i\,{\rm Tr}(SZ)},
\end{equation}
where $Z$ is a point in the Siegel upper half space $\HH_2$, and where the sum is taken over the set $\mathcal P_2$ of semi-integral\footnote{Recall that a matrix is semi-integral if its off-diagonal entries are in $\frac12\Z$, and its diagonal entries are in $\Z$.}, symmetric and positive semidefinite $2\times2$-matrices $S$. We say that $F$ satisfies the \emph{Maass relations} if, for all $S=\mat{a}{b/2}{b/2}{c}$,
 \begin{equation}\label{maass-rels-eqn2}
  a(\mat{a}{\frac b2}{\frac b2}{c}) = \sum\limits_{r | {\rm gcd}(a,b,c)} r^{k-1}\, a(\mat{\frac{ac}{r^2}}{\frac{b}{2r}}{\frac{b}{2r}}{1}).
 \end{equation}
Such a relation was first known to be satisfied by Eisenstein series; see \cite{ResSal}. Maass, in \cite{Ma}, started a systematic investigation of the space of modular forms satisfying these relations, calling this space the \emph{Spezialschar}. Within a few years it was proven, through the efforts of Maass, Andrianov and Zagier, that the Spezialschar is precisely the space of modular forms spanned by \emph{Saito-Kurokawa liftings}. Recall that a Saito-Kurokawa lifting is a Siegel modular form of weight $k$ constructed from an elliptic modular form of weight $2k-2$ with $k$ even. The book \cite{EZ} gives a streamlined account of the construction of these liftings, and of the proof that they span the \emph{Spezialschar}.

In addition to this classical approach, it is possible to construct Saito-Kurokawa liftings using automorphic representations theory. For simplicity, we only consider cuspforms. The procedure may be illustrated as follows:
\begin{equation}\label{SKdiagrameq}
 \begin{CD}
  \pi&\qquad&\PGL_2(\A)&@<{\rm Wald}<<&\meta(\A)&@>{\theta}>>\PGSp_4(\A)&\qquad&\Pi\\
  @AAA&&&&&&&&&@VVV\\
  f&&&&&&&&&&&F
 \end{CD}
\end{equation}
Here, $f$ is an elliptic cuspform of weight $2k-2$ with $k$ even and with respect to $\SL_2(\Z)$. Assuming that $f$ is an eigenform, it can be translated into an adelic function which generates a cuspidal automorphic representation $\pi$ of $\GL_2(\A)$. This representation has trivial central character, so really is a representation of $\PGL_2(\A)$. Since $\PGL_2\cong\SO_3$, there is a theta correspondence between this group and $\meta(\A)$, the double cover of $\SL_2(\A)$. More precisely, one considers the \emph{Waldspurger lifting} from $\meta(\A)$ to $\PGL_2(\A)$, which is a variant of the theta correspondence. Let $\tau$ be any pre-image of $\pi$ under this lifting\footnote{In this classical situation, there is precisely one pre-image. For modular forms with level, there can be several possibilities for $\tau$, leading to the phenomenon that one elliptic cusp form $f$ may have several Saito-Kurokawa liftings. We refer to \cite{Sch2} for more explanation.}. Since $\PGSp_4\cong\SO_5$, there is another
theta
correspondence between $\meta(\A)$ and $\PGSp_4(\A)$. We can use it to forward $\tau$ to an automorphic representation $\Pi$ of $\PGSp_4(\A)$. From this $\Pi$ one can extract a Siegel modular form $F$ of weight $k$. It turns out that $\Pi$ is cuspidal, so that $F$ is a cuspform. This $F$ coincides with the classical Saito-Kurokawa lifting of $f$. For the details of this construction, see \cite{PS-SK} and \cite{Sch}.

There is a marked difference between the classical and the representation theoretic constructions. The classical construction directly provides the Fourier coefficients of the modular form $F$ (in terms of the Fourier coefficients of the half-integral weight modular form corresponding to $f$ via the Shimura correspondence). In contrast, the representation theoretic Saito-Kurokawa lifting consists of the following statement: ``For each cuspidal elliptic eigenform $f$ of weight $2k-2$ with even $k$ there exists a cuspidal Siegel eigenform $F$ of weight $k$ such that its spin $L$-function is given by $L(s,F)=L(s,f)\zeta(s-k+1)\zeta(s-k+2)$.'' In this case the Fourier coefficients of $F$ are not readily available.

At the very least, one would like to know that the Fourier coefficients of the modular form $F$ constructed in the representation theoretic way satisfy the Maass relations. One quick argument consists in referring to either \cite{E} or \cite{O}. In these papers it is proven that the Fourier coefficients of $F$ satisfy the Maass relations if and only if $L(s,F)$ has a pole. The pole condition is satisfied because of the appearance of the zeta factors above.

However, it would be desirable to deduce the Maass relations directly from the representation theoretic construction. One reason is that this construction opens the way to generalizations in various directions, and for these more general situations results similar to \cite{E} or \cite{O} are not available. For example, what happens if we replace the above condition ``$k$ even'' by ``$k$ odd''? In this case it turns out that one can still do the representation theoretic construction, the difference being that the archimedean component of the automorphic representation $\Pi$ is no longer in the holomorphic discrete series. Hence, one will obtain a certain type of non-holomorphic Siegel modular form whose adelization generates a global CAP representation. As far as we know, the full details of this construction have yet to
be written out (though see~\cite{Mi}).
But since the non-archimedean situation is no different from the case
for even $k$, we expect this
new type of Siegel modular form to admit a Fourier expansion for which
the Maass relations
hold as well. One could prove such a statement if one had a direct
representation-theoretic way of deducing the Maass relations.
Similarly, we expect that a representation-theoretic proof of the
Maass relations would easily generalize to the case of Saito-Kurokawa lifts with
respect to congruence subgroups.

It was shown in \cite{O13} that a representation theoretic method for proving the Maass relations exists. In the present paper we take a similar, but slightly different approach. Common to both approaches is the fact that certain local Jacquet modules are one-dimensional. This may be interpreted as saying that the local representations in question admit a unique Bessel model, and this Bessel model is \emph{special} (see Sect.~\ref{besselsec} for precise definitions). While \cite{O13} makes use of certain Siegel series to derive an explicit formula for local $p$-adic Bessel functions, we use Sugano's formula, to be found in \cite{Su}.

Our main local result, Theorem \ref{localmaasstheorem} below, asserts the equivalence of five conditions on a given irreducible, admissible, spherical representation $\pi$ of $\GSp_4(F)$ with a special Bessel model (where $F$ is a $p$-adic field). The first condition is that one of the Satake parameters of $\pi$ is $q^{1/2}$ (where $q$ is the cardinality of the residue class field); in particular, such representations are non-tempered. The second condition is that $\pi$ is a certain kind of degenerate principal series representation; these representations occur in global CAP representations with respect to the Borel or Siegel parabolic subgroup. The third and fourth conditions are formulas relating certain values of the spherical Bessel function; the third formula is a local analogue of the Maass relations. The fifth condition is an explicit formula for some of the values of the spherical Bessel function; this formula is very similar to one appearing in \cite{K} and \cite{O13} for the values of a Siegel
series.

In the global part of this paper, we will explain how the classical Maass relations follow from this local result. It is known that the local components of the automorphic representation $\Pi$ in the diagram \eqref{SKdiagrameq} are of the kind covered in Theorem \ref{localmaasstheorem}. Hence, the corresponding spherical Bessel functions satisfy the ``local Maass relations'' (this implication is all that is needed from Theorem \ref{localmaasstheorem}). Since Bessel models are closely related to Fourier coefficients, one can deduce the global (classical) Maass relations from the local relations. To make this work one has to relate the classical notions with the representation theoretic concepts. While this is standard, some care has to be taken, which is why we carry these arguments out in some detail. In fact, we give two different proofs of the classical Maass relations; one uses a result proved by the second author in collaboration with Kowalski and Tsimerman~\cite{KST}, while the other relies on some
explicit computations with Bessel functions which may be of independent interest. As explained above, what we have in mind are future applications to more general situations. Finally, in the last section, we prove a result (Theorem~\ref{t:newsk}) that gives a new characterization of Saito-Kurokawa lifts involving a certain average of Fourier coefficients.

\vspace{3ex}\noindent
{\bf Notation}

\vspace{2ex}\noindent
Let $G=\GSp_4$ be the group of symplectic similitudes of semisimple rank $2$, defined by
$$
 \GSp_4 = \{ g \in \GL_4 : {}^tg J g = \mu(g) J, \,\,\mu(g) \in \GL_1\},  \text{ where } J = \mat{0}{1_2}{-1_2}{0}.
$$
Here, $\mu$ is called the similitude character. Let $\SSp_4 = \{g \in \GSp_4 : \mu(g) = 1\}$.  The Siegel parabolic subgroup $P$ of $\GSp_4$ consists of matrices whose lower left $2\times2$-block is zero. Its unipotent radical $U$ consists of all elements of the block form $\mat{1}{X}{}{1}$, where $X$ is symmetric. The standard Levi component $M$ of $P$ consists of all elements $\mat{A}{}{}{u\,^t\!A^{-1}}$ with $u\in\GL_1$ and $A\in\GL_2$.

Over the real numbers, we have the identity component $G(\R)^+:=\{g\in\GSp_4(\R):\mu(g)>0\}$.  Let $\HH_2$ be the Siegel upper half space of degree $2$. Hence, an element of $\HH_2$ is a symmetric, complex $2\times 2$-matrix with positive definite imaginary part. The group $G(\R)^+$ acts on $\HH_2$ via $g\langle Z \rangle = (AZ+B)(CZ+D)^{-1}$ for $g = \mat{A}{B}{C}{D}$.

Given any commutative ring $R$, we denote by ${\rm Sym}_2(R)$ the set of symmetric $2\times2$-matrices with coefficients in $R$. The symbol $\mathcal{P}_2$ denotes the set of positive definite, half-integral symmetric $2\times2$-matrices.
\section{Spherical Bessel functions}\label{besselsec}
In this section only, $F$ is a local non-archimedean field with ring of integers $\OF$, prime ideal $\p$, uniformizer $\varpi$ and order of residue field $q$. An irreducible, admissible representation of $\GSp_4(F)$ is called \emph{spherical} if it admits a spherical vector, i.e., a non-zero $\GSp_4(\OF)$-invariant vector. Let $(\pi,V)$ be such a representation. Then $\pi$ is a constituent of a representation parabolically induced from a character $\gamma$ of the standard Borel subgroup of $\GSp_4(F)$. The numbers
%Let $(\pi,V)$ be an irreducible admissible representation of $\GSp_4(F)$ of Type IIb (see Table A1 of \cite{RS}). Then $\pi$ is a constituent of a representation induced from an unramified character $\gamma$ of the standard Borel subgroup of $\GSp_4(F)$ of the following form.
% \begin{align}
% & \gamma^{(1)} = \gamma(\begin{bmatrix}\varpi\\&\varpi\\&&1\\&&&1\end{bmatrix}) = \alpha, \qquad  \gamma^{(2)} = \gamma(\begin{bmatrix}\varpi\\&1\\&&1\\&&&\varpi\end{bmatrix}) = q^{1/2}, \nonumber\\
% &\gamma^{(3)} = \gamma(\begin{bmatrix}1\\&1\\&&\varpi\\&&&\varpi\end{bmatrix}) = \alpha^{-1}, \qquad \gamma^{(4)} = \gamma(\begin{bmatrix}1\\&\varpi\\&&\varpi\\&&&1\end{bmatrix}) = q^{-1/2}, \label{alpha-defn}
% \end{align}
\begin{align}
& \gamma^{(1)} = \gamma(\begin{bmatrix}\varpi\\&\varpi\\&&1\\&&&1\end{bmatrix}), \qquad  \gamma^{(2)} = \gamma(\begin{bmatrix}\varpi\\&1\\&&1\\&&&\varpi\end{bmatrix}), \nonumber\\
&\gamma^{(3)} = \gamma(\begin{bmatrix}1\\&1\\&&\varpi\\&&&\varpi\end{bmatrix}), \qquad \gamma^{(4)} = \gamma(\begin{bmatrix}1\\&\varpi\\&&\varpi\\&&&1\end{bmatrix}), \label{alpha-defn}
\end{align}
are called the \emph{Satake parameters} of $\pi$. The conjugacy class of ${\rm diag}(\gamma^{(1)},\gamma^{(2)},\gamma^{(3)},\gamma^{(4)})$ in $\GSp_4(\C)$ determines the isomorphism class of $\pi$.

Note that $\gamma^{(1)}\gamma^{(3)}=\gamma^{(2)}\gamma^{(4)}=\omega_\pi(\varpi)$, where $\omega_\pi$ is the central character of $\pi$. Hence, in the case of trivial central character, the Satake parameters are $\{\alpha^{\pm1},\beta^{\pm1}\}$ for some $\alpha,\beta\in\C^\times$. In this case we allow ourselves a statement like ``one of the Satake parameters of $\pi$ is $\alpha^{\pm1}$''.

In this work we will employ the notation of \cite{ST} and the classification of \cite{RS} for constituents of parabolically induced representations of $\GSp_4(F)$. According to Table A.10 of \cite{RS}, the spherical representations are of type I, IIb, IIIb, IVd, Vd or VId, and a representation of one these types is spherical if and only if the inducing data is unramified. Note that type IVd is comprised of one-dimensional representations, which are irrelevant for our purposes. Representations of type I are irreducible principal series representations, and they are the only generic spherical representations.

By \cite{Sch}, representations of type IIb occur as local components of the automorphic representations attached to classical Saito-Kurokawa liftings. Recall that these automorphic representations are CAP (cuspidal associated to parabolic) with respect to the Siegel parabolic subgroup; this property has been defined on p.\ 315 of \cite{PS-SK}, where it was called ``strongly associated to $P$''. One can show that representations of type Vd and VId occur as local components of automorphic representations which are CAP with respect to $B$, the Borel subgroup. By Theorem 2.2 of \cite{PS-SK}, $P$-CAP and $B$-CAP representations with trivial central character have a common characterization as being theta liftings from the metaplectic cover of $\SL_2$. We will see below that spherical representations of type IIb, Vd and VId with trivial central character have a common characterization in terms of their spherical Bessel functions.

We will briefly recall the notion of Bessel model; for more details see \cite{RS2} (and \cite{PS1} for the archimedean case). Let $\psi$ be a fixed character of $F$. Let $S=\mat{a}{b/2}{b/2}{c}$ with $a,b,c\in F$. Such a matrix defines a character $\theta=\theta_S$ of $U$, the unipotent radical of the Siegel parabolic subgroup, by
$$
 \theta(\mat{1}{X}{}{1})=\psi({\rm Tr}(S X)),\qquad X\in{\rm Sym}_2(F).
$$
Every character of $U$ is of this form for a uniquely determined $S$. From now on we will assume that $\theta$ is \emph{non-degenerate}, by which we mean that $S$ is invertible. If $-\det(S)\in F^{\times2}$ we set $L=F\oplus F$ and say that we are in the \emph{split case}. Otherwise  we set $L=F(\sqrt{-\det(S)})$ and say that we are in the \emph{non-split case}. Below we will use the \emph{Legendre symbol}
\begin{equation}\label{legendreeq}
 \Big(\frac{L}{F}\Big)=\begin{cases}
                        -1&\text{if $L/F$ is an unramified field extension},\\
                        0&\text{if $L/F$ is a ramified field extension},\\
                        1&\text{if $L=F\oplus F$}.
                       \end{cases}
\end{equation}
Let $T=T_S$ be the torus defined by $T=\{g\in\GL_2:\:^tgSg=\det(g)S\}$. Then $T(F)\cong L^\times$. We think of $T(F)$ embedded into $\GSp_4(F)$ via $g\mapsto\mat{g}{}{}{\det(g)\,^tg^{-1}}$. Then $T(F)$ is the identity component of the stabilizer of $\theta$ in the Levi component of the Siegel parabolic subgroup. We call the semidirect product $R = TU$ the \emph{Bessel subgroup} defined by $S$. Given a character $\Lambda$ of $T(F)$, the map $tu\mapsto\Lambda(t)\theta(u)$ ($t\in T(F)$, $u\in U(F)$) defines a character $\Lambda\otimes\theta$ of $R(F)$.

Now let $(\pi,V)$ be an irreducible, admissible representation of $\GSp_4(F)$. A \emph{$(\Lambda,\theta)$-Bessel functional} for $\pi$ is a non-zero element of ${\rm Hom}_R(V,\C_{\Lambda\otimes\theta})$. Equivalently, a $(\Lambda,\theta)$-Bessel functional for $\pi$ is a non-zero functional $\beta$ on $V$ satisfying $\beta(\pi(r)v) = (\Lambda\otimes\theta) (r) v$ for all $r \in R(F)$, $v \in V$. Given such a functional $\beta$, the corresponding \emph{Bessel model} for $\pi$ consists of the functions $B(g)=\beta(\pi(g)v)$, where $v\in V$. By Theorem 6.1.4 of \cite{RS2} every infinite-dimensional $\pi$ admits a Bessel functional for some choice of $\theta$ and $\Lambda$. The question of uniqueness is discussed in Sect.~6.3 of \cite{RS2}. Bessel models with $\Lambda=1$ are called \emph{special}.

In the case of spherical representations, one may ask about an explicit formula for the spherical vector in a $(\Lambda,\theta)$-Bessel model. Such a formula was given by Sugano in \cite{Su} (at the same time proving that such models are unique for spherical representations). In the case that $\Lambda$ is unramified, Sugano's formula is conveniently summarized in Sect.~(3.6) of \cite{Fu}. We recall the result. To begin, we assume that the elements $a,b,c$ defining the matrix $S$ satisfy the following \emph{standard assumptions}:
\begin{equation}\label{standardassumptionseq}
 \begin{minipage}{85ex}
  \begin{itemize}
   \item $a,b\in\OF$ and $c\in\OF^\times$.
   \item In the non-split case, $b^2-4ac$ is a generator of the discriminant of $L/F$. In the split case, $b^2-4ac\in\OF^\times$.
  \end{itemize}
 \end{minipage}
\end{equation}
This is not a restriction of generality; using some algebraic number theory, one can show that, after a suitable transformation $S\mapsto\lambda\,^t\!ASA$ with $\lambda\in F^\times$ and $A\in\GL_2(F)$, the standard assumptions are always satisfied. One consequence of \eqref{standardassumptionseq} is the decomposition
\begin{equation}\label{doub-cos-gl2-eq}
 \GL_2(F) = \bigsqcup_{m \geq 0} T(F)\mat{\varpi^m}{}{}{1} \GL_2(\OF);
\end{equation}
see Lemma 2-4 of \cite{Su}. In conjunction with the Iwasawa decomposition, this implies
\begin{equation}\label{doub-cos-eq}
 \GSp_4(F) = \bigsqcup_{\substack{l, m \in \Z \\ m \geq 0}} R(F)h(l,m) \GSp_4(\OF),
\end{equation}
where
$$
 h(l,m) = \begin{bmatrix}\varpi^{l+2m}\\&\varpi^{l+m}\\&&1\\&&&\varpi^m\end{bmatrix}.
$$
Hence, a spherical Bessel function $B$ is determined by the values $B(h(l,m))$. It is easy to see that $B(h(l,m))=0$ if $l<0$ (see Lemma (3.4.4) of \cite{Fu}). Sugano's formula now says that
\begin{equation}\label{suganosformulaeq}
 \sum_{l,m\geq0}B(h(l,m))x^my^l=\frac{H(x,y)}{P(x)Q(y)}
\end{equation}
where $P(x)$, $Q(y)$ and $H(x,y)$ are polynomials whose coefficients depend on the Satake parameters, on the value of $\big(\frac{L}{F}\big)$, and on $\Lambda$; see p.\ 205 of \cite{Fu} for details. The formula implies in particular that $B(1)\neq0$, so that we may always normalize $B(1)$ to be $1$.

% Assume that the non-degenerate matrix $S'\in{\rm Sym}_2(F)$ satisfies the conditions from \cite{Su}. Let $\beta:V\to\C$ be a non-zero functional with the property $\beta(\pi(n)v)=\Theta_{S'}(n)\beta(v)$ for all $v\in V$ and $n\in U(F)$. By Lemma 5.2.2 of \cite{RS2}, such a $\beta$ exists, is unique up to scalars, and satisfies $\beta(\pi(m)v)=\beta(v)$ for all $v\in V$ and $m\in T_{S'}$. In the special case $F=\Q_p$ we have already stated these facts in Lemma \ref{localBessellemma}. Thus, $\beta$ is a special Bessel functional. Let $v^0\in V$ be the spherical vector, and let $B(g)=\beta(\pi(g)v^0)$ be the corresponding Bessel function. We normalize $\beta$ such that $B(1)=1$. We will first reduce the evaluation of $B(h(l,m))$ to the evaluation of $B(h(0, \ast))$. We call this the local analogue of the Maass relation since the classical Maass relation for the Fourier coefficients $a(S)$ follows from this relation. Lastly, we will obtain an explicit formula for $B(h(0,m))$.

With these preparations, we may now formulate our main local theorem.

\begin{theorem}\label{localmaasstheorem}
 Let $\pi$ be an irreducible, admissible, spherical representation of $\GSp(4,F)$ with trivial central character. Assume that $\pi$ admits a special Bessel model with respect to the matrix $S$. Let $B$ be the spherical vector in such a Bessel model for $\pi$, normalized such that $B(1)=1$. Then the following are equivalent.
 \begin{enumerate}
  \item One of the Satake parameters of $\pi$ is $q^{\pm1/2}$.
  \item $\pi$ is one of the representations in the following list:
   \begin{itemize}
    \item $\chi1_{\GL(2)}\rtimes\chi^{-1}$ for an unramified character $\chi$ of $F^\times$ (type IIb).
    \item $L(\nu\xi,\xi\rtimes\nu^{-1/2})$, where $\xi$ is the non-trivial, unramified, quadratic character of $F^\times$ (type Vd).
    \item $L(\nu,1_{F^\times}\rtimes\nu^{-1/2})$ (type VId).
   \end{itemize}
  \item For all $l,m\geq0$,
   \begin{equation}\label{maass-rels-thmeq3c}
    B(h(l,m))=\sum_{i=0}^lq^{-i} B(h(0,l+m-i)).
   \end{equation}
  \item The following relation is satisfied:
    \begin{equation}\label{maass-rels-thmeq3d2}
    B(h(1,0))=B(h(0,1)) + q^{-1}.
   \end{equation}
  \item The following two conditions are satisfied:
   \begin{itemize}
    \item For all $m\geq0$,
     \begin{equation}\label{localmaasstheoremeq2}
      q^{3m/2} B(h(0,m)) = \frac{\alpha^{m+1}-\alpha^{-m-1}}{\alpha-\alpha^{-1}} -\Big(\frac{L}{F}\Big) q^{-1/2}\,\frac{\alpha^m-\alpha^{-m}}{\alpha-\alpha^{-1}},
     \end{equation}
     where $\alpha^{\pm1}$ is one of the Satake parameters of $\pi$. Here, if $\alpha=\alpha^{-1}$, we understand $\frac{\alpha^m-\alpha^{-m}}{\alpha-\alpha^{-1}}=\sum_{i=1}^m\alpha^{m+1-2i}$.
    \item We are not in the following exceptional situation: $\big(\frac{L}{F}\big)=-1$ (the inert case) and $\pi=\chi1_{\GL(2)}\rtimes\chi^{-1}\xi$ (type IIb), where $\chi,\xi$ are the unramified characters with $\chi(\varpi)=\pm i$ and $\xi(\varpi)=-1$.
   \end{itemize}
 \end{enumerate}
\end{theorem}
\begin{proof}
i) $\Leftrightarrow$ ii) follows by inspecting the list of Satake parameters of all spherical representations; see Table A.7 of \cite{RS}.

iii) $\Rightarrow$ iv) is trivial.

iv) $\Rightarrow$ i) $\Rightarrow$ iii) Observe that \eqref{maass-rels-thmeq3c} is equivalent to the following identity between generating series,
\begin{equation}\label{maass-rels-thmeq3d}
 \sum_{l,m\geq0}B(h(l,m))x^my^l=\sum_{l,m\geq0}\sum_{i=0}^lq^{-i} B(h(0,l+m-i))x^my^l.
\end{equation}
By Sugano's formula, the left hand side equals
$$
 \text{LHS}=\frac{H(x,y)}{P(x)Q(y)},
$$
with $H,P,Q$ as in Proposition 2-5 of \cite{Su}. For the right hand side of \eqref{maass-rels-thmeq3d}, we calculate
\begin{align*}
 \text{RHS}&=\sum_{m=0}^\infty\:\sum_{l=0}^\infty\:\sum_{i=0}^lq^{-i} B(h(0,l+m-i))x^my^l\\
 &=\sum_{m=0}^\infty\:\sum_{i=0}^\infty\:\sum_{l=i}^\infty q^{-i} B(h(0,l+m-i))x^my^l\\
 &=\sum_{m=0}^\infty\:\sum_{i=0}^\infty\:\sum_{l=0}^\infty q^{-i} B(h(0,l+m))x^my^{l+i}\\
 &=\frac1{1-q^{-1}y}\;\sum_{m=0}^\infty\:\sum_{l=0}^\infty B(h(0,l+m))x^my^l\\
 &=\frac1{1-q^{-1}y}\;\sum_{j=0}^\infty\:\sum_{l+m=j} B(h(0,j))x^my^l\\
 &=\frac1{1-q^{-1}y}\;\sum_{j=0}^\infty\:B(h(0,j))\,\frac{x^{j+1}-y^{j+1}}{x-y}\\
 &=\frac1{(1-q^{-1}y)(x-y)}\bigg(x\sum_{j=0}^\infty\:B(h(0,j))x^j-y\sum_{j=0}^\infty\:B(h(0,j))y^j\bigg)\\
 &=\frac1{(1-q^{-1}y)(x-y)}\bigg(x\frac{H(x,0)}{P(x)Q(0)}-y\frac{H(y,0)}{P(y)Q(0)}\bigg)\\
 &=\frac1{(1-q^{-1}y)(x-y)}\bigg(x\frac{H(x,0)}{P(x)}-y\frac{H(y,0)}{P(y)}\bigg).
\end{align*}
Hence, \eqref{maass-rels-thmeq3d} is equivalent to
\begin{equation}\label{maass-rels-thmeq3e}
 (1-q^{-1}y)(x-y)H(x,y)P(y)-Q(y)\Big(xP(y)H(x,0)-yP(x)H(y,0)\Big)=0.
\end{equation}
If one of the Satake parameters is $q^{\pm1/2}$, then one can verify that \eqref{maass-rels-thmeq3e} is satisfied. This shows that i) $\Rightarrow$ iii). Conversely, assume that iv) is satisfied.  Let $F(x,y)$ be the polynomial on the left hand side of \eqref{maass-rels-thmeq3e}. Then iv) is equivalent to saying that the coefficient of $y$ of the power series
$$
 T(y) = \frac{H(0,y)}{Q(y)} - \frac{H(y,0)}{(1-q^{-1}y)P(y)},
$$
which has no constant term, vanishes. In particular, this means that the $y^2$-coefficient of $F(0,y)=-y Q(y)P(y)(1-q^{-1}y)T(y)$ is equal to 0. But it can be easily checked that the $y^2$-coefficient of $F(0,y)$ is given by
$$
 \frac{(q^{-1/2}-\alpha^{-1})(q^{-1/2}-\alpha)(q^{-1/2}-\beta^{-1})(q^{-1/2}-\beta)}{q-\delta},
$$
where $\alpha^{\pm1}$, $\beta^{\pm1}$ are the Satake parameters of $\pi$. It follows that $\alpha=q^{\pm1/2}$ or $\beta=q^{\pm1/2}$. The completes the proof of iv) $\Rightarrow$ i).

i) $\Leftrightarrow$ v) Let $\alpha^{\pm1}$, $\beta^{\pm1}$ be the Satake parameters of $\pi$. By Sugano's formula,
\begin{equation}\label{eqn1}
P(x) \sum\limits_{m \geq 0} B(h(0,m)) x^m = H(x, 0).
\end{equation}
Note that $H(x,0)$ (resp.\ $P(x)$) is a degree $3$ (resp.\ degree $4$) polynomial. Write
$$
 P(x) = P_0 + P_1x + P_2 x^2 + P_3x^3 + P_4 x^4, \qquad H(x,0) = H_0 + H_1 x + H_2 x^2 + H_3 x^3.
$$
Abbreviating $A=\alpha+\alpha^{-1}$ and $B=\beta+\beta^{-1}$, as well as $\delta=\big(\frac{L}{F}\big)$, we have
\begin{align*}
 P_0&=1,\\
 P_1&=-q^{-2}AB,\\
 P_2&=q^{-4}(A^2+B^2-2),\\
 P_3&=-q^{-6}AB,\\
 P_4&=q^{-8},
\end{align*}
and
\begin{align*}
 H_0&=1,\\
 H_1&=\frac1{q^2(q-\delta)}\Big(q+1+\delta(\delta+1)-q^{1/2}(\delta+1)(A+B)+\delta AB\Big),\\
 H_2&=\frac1{q^5(q-\delta)}\Big(q(\delta+1)+\delta^2(q+1)-q^{1/2}\delta(\delta+1)(A+B)+\delta q AB\Big),\\
 H_3&=-q^{-7}\delta.
\end{align*}
Taking the $m^{th}$ derivative of both sides of \eqref{eqn1}, setting $x=0$ and dividing by $m!$, we get the following recurrence relation for $B(h(0,m))$,
\begin{equation}\label{recurrence-reln}
 \sum\limits_{i=0}^4 B(h(0, m-i)) P_i = H_m,
\end{equation}
where $H_m = 0$ if $m > 3$. Assume that the Satake parameters are $\alpha^{\pm1}$ and $q^{\pm1/2}$, i.e., $B=q^{1/2}+q^{-1/2}$. Then, using \eqref{recurrence-reln} and induction on $m$, it is easy to verify that \eqref{localmaasstheoremeq2} holds. Conversely, assume that \eqref{localmaasstheoremeq2} holds for all $m\geq0$. Then, from \eqref{recurrence-reln} for $m=1$, we find that either $B=q^{1/2}+q^{-1/2}$ or $A=q^{-1/2}(\delta+1)$. Assume that $A=q^{-1/2}(\delta+1)$. From \eqref{recurrence-reln} for $m=2,3,4$, we conclude, after some calculation, that $B=-q^{1/2}-q^{-1/2}$ and $\delta=-1$. It follows that $A=0$, so that $\alpha=\pm i$. Looking at Satake parameters, this is precisely the excluded exceptional situation.
\end{proof}

\noindent{\bf Remarks:}
 \begin{enumerate}
  \item The second condition in part v) of this theorem cannot be omitted, since in this exceptional situation the formula \eqref{localmaasstheoremeq2} holds as well.
  \item There is a certain analogy of the identity \eqref{maass-rels-thmeq3c} with the classical Maass relations \eqref{maass-rels-eqn2}. In fact, in the proof of Theorem \ref{maass-rels-thm} we will show that the Maass relations are implied by the local relations \eqref{maass-rels-thmeq3c}.
  \item Combining the formulas iii) and v) of Theorem \ref{localmaasstheorem}, we obtain
   \begin{equation}\label{siegelserieseq}
    q^{3m/2}B(h(l,m-l))=\sum_{i=0}^lq^{i/2}\bigg(\frac{\alpha^{m-i+1}-\alpha^{-(m-i+1)}}{\alpha-\alpha^{-1}} -\Big(\frac{L}{F}\Big) q^{-1/2}\,\frac{\alpha^{m-i}-\alpha^{-(m-i)}}{\alpha-\alpha^{-1}}\bigg)
   \end{equation}
   for $m\geq l\geq0$. The expression on the right hand side, viewed as a polynomial in $\alpha$, is related to the value of a certain Siegel series; see Hilfssatz 10 in \cite{K} and Corollary 5.1 in \cite{O13}. In fact, formula \eqref{siegelserieseq} appears as Lemma 8.1 of \cite{O13}.
 \end{enumerate}
\section{Adelization and Fourier coefficients}
We turn to classical Siegel modular forms and their adelization. Let $\Gamma = \Symp_4(\Z)$ and $S_k^{(2)}(\Gamma)$ be the space of holomorphic cuspidal Siegel modular forms of degree $2$ and weight $k$ with respect to $\Gamma$. Hence, if $F \in S_k^{(2)}(\Gamma)$, then for all $\gamma \in \Gamma$ we have $F |_k \gamma=F$, where
\begin{equation}\label{Siegel mod form}
    (F |_k g)(Z) := \mu(g)^k j(g, Z)^{-k}
    F(g \langle Z \rangle)
\end{equation}
for $g \in G(\R)^+$ and $Z \in \HH_2$, the Siegel upper half space. Here $j(g, Z) = \det(CZ+D)$ for $g = \mat{A}{B}{C}{D} \in G(\R)^+$. The Fourier expansion of $F$ is given by
\begin{equation}\label{Four-exp}
 F(Z) = \sum\limits_S a(S) e^{2 \pi i\,{\rm Tr}(SZ)},
\end{equation}
where the sum is taken over the set $\mathcal P_2$ of semi-integral, symmetric and positive definite matrices $S$.

Let $\A$ be the ring of adeles of $\Q$. It follows from the strong approximation theorem for $\SSp_4$ that
\begin{equation}\label{spin strong approx}
    G(\A) = G(\Q) G(\R)^+ K_0
\end{equation}
where $K_0 := \prod\limits_{p < \infty} \Gamma_p$ with $\Gamma_p=G(\Z_p)$. Let $F \in S_k^{(2)}(\Gamma)$. Write $g \in G(\A)$ as $g = g_{\Q} g_{\infty}g_0$ with $g_{\Q} \in G(\Q)$, $g_{\infty} \in G(\R)^+$, $g_0 \in K_0$, and define $\Phi_F : G(\A)
\rightarrow \C$ by the formula
\begin{equation}\label{adelic automorphic defn}
    \Phi_F(g) := (F |_k g_{\infty})(i 1_2).
\end{equation}
Since $G(\R)^+K_0 \cap G(\Q) = \Gamma$, the function $\Phi_F$ is well-defined. From the definition it is clear that for all $g \in G(\A)$, $\rho \in G(\Q)$, $k_0 \in K_0$, $k_{\infty} \in K_{\infty}$ and $z \in Z(\A)$, it satisfies
\begin{equation}\label{adelic-fn-props}
\Phi_F(z \rho g k_{\infty} k_0) = \Phi_F(g) j(k_{\infty}, i 1_2)^{-k}.
\end{equation}
Here $Z \simeq \GL_1$ is the center of $\GSp_4$ and $K_{\infty}\simeq U(2)$ is the standard maximal compact subgroup of $\Symp_{4}(\R)$.

Let $P=MU$ be the Siegel parabolic subgroup of $G$. By the Iwasawa decomposition, $G(\A) = U(\A)M(\A)K_{\infty}K_0$. Let $\psi : \Q \backslash \A \rightarrow \C^\times$ be the character such that $\psi(x) = e^{2 \pi i x}$ if $x \in \R$ and $\psi(x) = 1$ for $x \in \Z_p$. Given $S\in{\rm Sym}_2(\Q)$, one obtains a character $\Theta_{S}$ of $U(\Q) \backslash U(\A)$ by
$$
 \Theta_{S}(\mat{1}{X}{}{1}) = \psi({\rm Tr}(S X)).
$$
Note that every character of $U(\Q) \backslash U(\A)$ is obtained in this way. For $S \in {\rm Sym}_2(\Q)$  we define the following adelic Fourier coefficient of $\Phi_F$,
\begin{equation}\label{Adelic-fourier-defn}
    \Phi_F^{S}(g) := \int\limits_{U(\Q) \backslash U(\A)}
    \Phi_F(n g) \Theta_{S}^{-1}(n)\,dn \qquad\text{ for } g \in G(\A).
\end{equation}
The following result, which is standard, provides a formula for $\Phi_F^{S}(g)$ in terms of the Fourier coefficients of $F$.
\begin{proposition}\label{fourier-relation-prop}
 Let $g = n_0mk_\infty k_0$, with $n_0 \in U(\A)$, $m \in M(\A)$, $k_\infty \in K_\infty$ and $k_0 \in K_0$. Let $m = m_\Q m_\infty m_0$, with $m_\Q \in M(\Q)$, $m_\infty \in M(\R)^+$ and $m_0 \in M(\A) \cap K_0$. Write $m_\Q = \mat{A}{}{}{v^{-1}\,^t\!A^{-1}}$ with $A \in \GL_2(\Q)$ and $v \in \Q^\times$. Let $Z_0 = m_\infty \langle i1_2 \rangle$. Let $S \in {\rm Sym}_2(\Q)$ be non-degenerate, and let $S' = v\,^t\!A S A$.  Let $F \in S_k^{(2)}(\Gamma)$ with Fourier expansion (\ref{Four-exp}) and let $\Phi_F^S$ be as defined in (\ref{Adelic-fourier-defn}). Then
 \begin{equation}\label{fourier-relation-eqn}
  \Phi_F^{S}(g) = \begin{cases}
    \Theta_{S}(n_0) \mu(m_\infty)^k\,j(g_\infty, i1_2)^{-k}\,a(S') e^{2 \pi i\,{\rm Tr}(S'Z_0)} & \text{ if } S' \in \mathcal P_2,\\
    0 & \text{ otherwise},
  \end{cases}
 \end{equation}
 where $g_\infty=m_\infty k_\infty$. In particular, if $S \in \mathcal P_2$, we have
 \begin{equation}\label{Phi-A-reln}
  \Phi_F^{S}(1) = a(S) e^{-2 \pi\,{\rm Tr}(S)}.
 \end{equation}
 Also, for $S \in {\rm Sym}_2(\Q)$ and $S' = v\,^t\!A S A$, with $A \in \GL_2(\Q)$ and $v \in \Q^\times$, we have
 \begin{equation}\label{change-of-T}
  \Phi_F^{S'}(g) = \Phi_F^{S}(\mat{A}{}{}{v^{-1}\,^t\!A^{-1}}g) \qquad \text{ for all } g \in G(\A).
 \end{equation}
\end{proposition}
\begin{proof}
From (\ref{adelic-fn-props}) and the definition (\ref{Adelic-fourier-defn}), we get
$$
 \Phi_F^{S}(n_0mk_\infty k_0) = \Theta_{S}(n_0) j(k_{\infty}, i 1_2)^{-k}\,\Phi_F^{S}(m_\Q m_\infty).
$$
A change of variable $n \mapsto m_\Q n m_\Q^{-1}$ in (\ref{Adelic-fourier-defn}) gives
$$
 \Phi_F^{S}(m_\Q m_\infty) = \Phi_F^{S'}(m_\infty),
$$
where $S' = v\,^t\!A S A$. For every prime $p$, let $\mathfrak{a}_p \subset \Z_p$ be the largest ideal such that whenever $n \in\prod\limits_{p < \infty} U(\mathfrak{a}_p)$ then $\Theta_{S'}(n) = 1$. Let $t = \prod\limits_p p^{r_p} \in \Z$, where $\mathfrak{a}_p = p^{r_p} \Z_p$. Note that $t=1$ if and only if $S'$ is semi-integral. Since
$$
 ({\rm Sym}_2(\Z)\backslash{\rm Sym}_2(\R))\times\prod\limits_{p<\infty}{\rm Sym}_2(\Z_p)
$$
is a fundamental domain for $U(\Q)\backslash U(\A)$ (where we identify $U\cong{\rm Sym_2})$, so is
$$
 (t\,{\rm Sym}_2(\Z)\backslash{\rm Sym}_2(\R))\times\prod\limits_{p<\infty}{\rm Sym}_2(\mathfrak{a}_p).
$$
Therefore,
\begin{align*}
 \Phi_F^{S'}(m_\infty) &= \bigg(\prod\limits_{p < \infty}\:\int\limits_{U(\mathfrak{a}_p)}1\,dn_p\bigg)
  \int\limits_{t\,{\rm Sym}_2(\Z)\backslash{\rm Sym}_2(\R)} \Phi_F(nm_\infty) \Theta_{S'}^{-1}(n)\,dn \\
 &= \bigg(\prod\limits_{p < \infty}\:\int\limits_{U(\mathfrak{a}_p)}1\,dn_p\bigg)
   \sum\limits_{\eta\in t\,{\rm Sym}_2(\Z)\backslash{\rm Sym}_2(\Z)}\;\int\limits_{{\rm Sym}_2(\Z)\backslash{\rm Sym}_2(\R)} \Phi_F(\eta nm_\infty) \Theta_{S'}^{-1}(\eta n)\,dn \\
 &= \bigg(\prod\limits_{p < \infty}\:\int\limits_{U(\mathfrak{a}_p)}1\,dn_p\bigg)
  \bigg(\sum\limits_{\eta\in t\,{\rm Sym}_2(\Z)\backslash{\rm Sym}_2(\Z)}\Theta_{S'}^{-1}(\eta)\bigg) \int\limits_{U(\Z) \backslash U(\R)} \Phi_F(nm_\infty) \Theta_{S'}^{-1}(n)\,dn.
\end{align*}
Since $\Theta_{S'}^{-1}$ is a non-trivial character of $tU(\Z) \backslash U(\Z)$ if $t > 1$, the sum is zero unless $t=1$. Hence, assume that $t=1$, i.e., $\mathfrak{a}_p = \Z_p$ for all $p$.  By \eqref{adelic automorphic defn},
\begin{align*}
 \Phi_F^{S'}(m_\infty) &=   \int\limits_{{\rm Sym_2(\Z)} \backslash {\rm Sym_2(\R)}} (F |_k \big(\mat{1}{X}{}{1} m_\infty\big))(i1_2) \Theta_{S'}^{-1}(\mat{1}{X}{}{1})\,dX\\
 &= \mu(m_\infty)^k j(m_\infty, i1_2)^{-k} \int\limits_{{\rm Sym_2(\Z)} \backslash {\rm Sym_2(\R)}} F(Z_0 + X) e^{-2 \pi i\,{\rm Tr}(S'X)}\,dX.
\end{align*}
Now, substitute the Fourier expansion of $F$ to get
\begin{align*}
 \Phi_F^{S'}(m_\infty) &= \mu(m_\infty)^k j(m_\infty, i1_2)^{-k} \sum\limits_{T \in \mathcal P_2} a(T) e^{2 \pi i\,{\rm Tr}(TZ_0)} \\
 & \qquad \qquad \times \int\limits_{{\rm Sym_2(\Z)} \backslash {\rm Sym_2(\R)}} e^{2 \pi i\, {\rm Tr}(TX)} e^{-2 \pi i\,{\rm Tr}(S'X)}\,dX.
\end{align*}
If $S' \not\in \mathcal P_2$, then the integral above is zero for every $T$, otherwise it is equal to one exactly for $T = S'$. This proves (\ref{fourier-relation-eqn}). One can obtain (\ref{change-of-T}) by a simple change of variables.
\end{proof}
\section{Special automorphic forms}\label{specialautformsec}
We now assume that $F \in S_k^{(2)}(\Gamma)$ is a Hecke eigenform and is a Saito-Kurokawa lift of $f \in S_{2k-2}(\SL_2(\Z))$, with $k$ even, as in \S6 of \cite{EZ}. Let $\Phi_F$ be as defined in (\ref{adelic automorphic defn}) and let $(\pi_F, V_F)$ be the irreducible cuspidal automorphic representation of $G(\A)$ generated by right translates of $\Phi_F$. Then $\pi_F$ is isomorphic to a restricted tensor product $\otimes_{p\leq\infty}\pi_p$ with irreducible, admissible representations $\pi_p$ of $\GSp_4(\Q_p)$. The following is well-known (see, for example, \cite{Sch}):
\begin{itemize}
 \item The archimedean component $\pi_\infty$ is a holomorphic discrete series representation with scalar minimal $K$-type determined by the weight $k$. Following the notation of \cite{PS1}, we denote this representation by $\mathcal{E}(k,k)$.
 \item For a prime number $p$, the representation $\pi_p$ is a degenerate principal series representation $\chi1_{\GL(2)}\rtimes\chi^{-1}$ with an unramified character $\chi$ of $\Q_p^\times$. Here, we are using the notation of $\cite{RS}$. In particular, $\pi_p$ is a representation of type IIb according to Table A.1 of \cite{RS}. Note that these are non-tempered, non-generic representations.
\end{itemize}
\begin{lemma}\label{localBessellemma}
 Let $p$ be a prime number or $p=\infty$. Let $S\in{\rm Sym}_2(\Q_p)$ be non-degenerate. In the archimedean case, assume also that $S$ is positive or negative definite. Let $T_S$ be the connected component of the stabilizer of the character $\Theta_S$ of $U(\Q_p)$. Explicitly,
 \begin{equation}\label{localBessellemmaeq0}
  T_S=\{g\in\GL_2:\:^tgSg=\det(g)S\},
 \end{equation}
 where we embed $\GL_2$ into $\GSp_4$ via $g\mapsto\mat{g}{}{}{\det(g)\,^tg^{-1}}$.
 Let $V_p$ be a model for $\pi_p$, and consider functionals $\beta_p:\:V_p\to\C$ with the property
 \begin{equation}\label{localBessellemmaeq1}
  \beta_p(\pi_p(n)v)=\Theta_S(n)\beta_p(v)\qquad\text{for all }v\in V_p\text{ and }n\in U(\Q_p).
 \end{equation}
 Then:
 \begin{enumerate}
  \item The space of such functionals $\beta_p$ is one-dimensional.
  \item If $\beta_p$ satisfies \eqref{localBessellemmaeq1}, then it automatically satisfies
   \begin{equation}\label{localBessellemmaeq2}
    \beta_p(\pi_p(m)v)=\beta_p(v)\qquad\text{for all }v\in V_p\text{ and }m\in T_S(\Q_p).
   \end{equation}
 \end{enumerate}
\end{lemma}
\begin{proof}
This follows from Lemma 5.2.2 of \cite{RS2} in the non-archimedean case, and from Theorem 3.10 of \cite{PS} in the archimedean case.
\end{proof}

In the language of Bessel models, Lemma \ref{localBessellemma} states that the only such model admitted by $\pi_p$ is \emph{special}, i.e., with trivial character on $T_S(\Q_p)$; see Sect.\ \ref{besselsec} for the definition of Bessel models in the non-archimedean case, and \cite{PS}, Sect.~2.6, for the definition in the archimedean case. Part i) of Lemma \ref{localBessellemma} asserts the uniqueness of such models. We remark that property ii) in this lemma is precisely the ``$U$-property'' of \cite{PS}.

We fix a distinguished vector $v^0_p\in V_p$ for each of our local representations $\pi_p$: If $p$ is finite, we let $v^0_p$ be a spherical (i.e., non-zero $G(\Z_p)$-invariant) vector, and if $p=\infty$ we let $v^0_p$ be a vector spanning the one-dimensional $K_\infty$-type determined by $k$. Note that the construction of the restricted tensor product $\otimes V_p$ depends on the choice of distinguished vectors almost everywhere, and we use the vectors $v^0_p$ for this purpose.

Let $S\in{\rm Sym}_2(\Q)$ be positive or negative definite. We will see in a later section that $\beta_p(v^0_p)\neq0$ for almost all $p$. For those places where this is the case, we normalize the $\beta_p$ such that $\beta_p(v^0_p)=1$.

The following lemma states that the automorphic forms in the space of $\pi$ are \emph{special} in the sense of \cite{PS}, p.\ 310.
\begin{lemma}\label{extra-inv-lem}
 Let the notations be as above. Then, for any non-degenerate $S\in{\rm Sym}_2(\Q)$, and all $\Phi\in V_F$, we have
 \begin{equation}\label{also-bessel-eqn}
  \Phi^{S}(mg) = \Phi^{S}(g) \qquad \text{ for all } g \in G(\A) \text{ and } m \in T_S(\A).
 \end{equation}
\end{lemma}
\begin{proof}
We fix $S$. The assertion is trivial if the functional
$$
 \beta:\:V\longrightarrow\C,\qquad\beta(\Phi):=\Phi^S(1)=\int\limits_{U(\Q)\backslash U(\A)}\Phi(n)\Theta_S^{-1}(n)\,dn
$$
is zero. Assume that $\beta$ is non-zero. For each place $p$, let $V_p$ be a model for $\pi_p$. By a standard argument, $\beta$ induces a non-zero functional $\beta_p:\:V_p\to\C$ with the property \eqref{localBessellemmaeq1}. Looking at the archimedean place, Corollary 4.2 of \cite{PS1} implies that $S$ is positive or negative definite. By the uniqueness asserted in Lemma \ref{localBessellemma}, it follows that there exists a non-zero constant $C_S$ such that
\begin{equation}\label{extra-inv-lemeq1}
 \beta(\Phi)=C_S\prod\limits_{p\leq\infty}\beta_p(v_p),
\end{equation}
whenever $\Phi\in V_F$ corresponds to the pure tensor $\otimes v_p$ via $V_F\cong\otimes V_p$; note that the product on the right is finite by our normalizations. Using ii) of Lemma \ref{localBessellemma}, it follows that $\beta(\pi(m)\Phi)=\beta(\Phi)$ for all $m\in T_S(\A)$. Since $\Phi$ is arbitrary, this implies the assertion of the lemma.
\end{proof}
\section{A proof of the classical Maass relations}
For $S = \mat{a}{b/2}{b/2}{c}\in \mathcal{P}_2$ we denote $c(S)=\gcd(a,b,c)$ (resp.\ $\disc(S)=-4\det(S)$) and call $c(S)$ the content (resp., call $\disc(S)$ the discriminant) of $S$. For $S_1, S_2 \in \mathcal{P}_2$, we say that $S_1\sim S_2$ if there exists a matrix $A \in\SL_2(\Z)$ such that $^t\!A S_1 A = S_2$. For any $S \in \mathcal{P}_2$, let $[S]$ denote the equivalence class of $S$ under the above relation; note that all matrices in a given equivalence class have the same content and discriminant. For any discriminant $D < 0$ (recall that a discriminant is an integer congruent to 0 or 1 modulo 4) and any positive integer $L$, we let $H(D; L)$ denote the set of equivalence classes of matrices in $\mathcal{P}_2$ whose content is equal to $L$ and whose discriminant is equal to $DL^2$. In particular, if $S \in \mathcal{P}_2$, then $[S] \in H(D; L)$ where $L=c(S)$ and $D=\disc(S)/c(S)^2$. It is clear that the map $[S] \mapsto [LS]$ gives a bijection of sets $H(D;1) \rightarrow H(D;L)$.

Our objective in this section is to prove the following theorem.

\begin{theorem}\label{t:globalmain}
 Let $F$ be a cuspidal Siegel Hecke eigenform of weight $k$ with respect to $\Symp_4(\Z)$. For $S \in \mathcal P_2$, let $a(S)$ denote the Fourier coefficient of $F$ at $S$. Suppose that $F$ is a Saito-Kurokawa lift.  Then the following hold.
 \begin{enumerate}
  \item Assume that $\disc(S_1)=\disc(S_2)$ and $c(S_1)=c(S_2)$. Then $a( S_1)=a(S_2)$.
  \item \label{part2} For $D<0$ a discriminant and $L$ a positive integer, define $a(D;L)=a( S)$ where $S$ is any member of $\mathcal P_2$ satisfying $c(S)=L$, $\disc(S)=DL^2$; this is well-defined by the previous part. Then the following relation holds:
   \begin{equation}\label{maass-rels-eqn2b}
    a(D;L) = \sum\limits_{r | L} r^{k-1}\,a(D(L/r)^2;1).
   \end{equation}
 \end{enumerate}
\end{theorem}

\begin{corollary}[Maass relations]
Let $F$ be a cuspidal Siegel Hecke eigenform of weight $k$ with respect to $\Symp_4(\Z)$ which is a Saito-Kurokawa lift. For $S \in \mathcal P_2$, let $a(S)$ denote the Fourier coefficient of $F$ at $S$. Then
 \begin{equation}\label{maass-rels-eqn-3}
  a(\mat{a}{\frac b2}{\frac b2}{c}) = \sum\limits_{r | {\rm gcd}(a,b,c)} r^{k-1} a(\mat{\frac{ac}{r^2}}{\frac{b}{2r}}{\frac b{2r}}{1}).
 \end{equation}
\end{corollary}
\begin{proof}
In the notation of the above Theorem, $$a(\mat{a}{\frac b2}{\frac b2}{c}) = a((b^2-4ac)/{\rm gcd}(a,b,c)^2;{\rm gcd}(a,b,c)), \quad a(\mat{\frac{ac}{r^2}}{\frac{b}{2r}}{\frac b{2r}}{1})=a((b^2-4ac)/r^2;1).$$ Now the corollary follows immediately from~\eqref{maass-rels-eqn2b}.
\end{proof}

Let us now prove Theorem~\ref{t:globalmain}. As a first step, we recall a very useful characterization of the elements of $H(D;L)$. Let $d<0$ be a \emph{fundamental discriminant}\footnote{Recall that an integer $n$ is called a fundamental discriminant if $n$ is the discriminant of the field $\Q(\sqrt{n})$. This is equivalent to saying that \emph{either} $n$ is a squarefree integer congruent to 1 modulo 4 \emph{or} $n = 4m$ where $m$ is a squarefree integer congruent to 2 or 3 modulo 4.}. We define $S_d \in \mathcal{P}_2$ as follows.
\begin{equation}\label{defsd}
 S_d = \begin{cases}
  \mat{\frac{-d}{4}}{0}{0}{1}& \text{ if } d\equiv 0\pmod{4}, \\[4ex]
  \mat{\frac{1-d}{4}}{\frac12}{\frac12}{1}& \text{ if } d\equiv 1\pmod{4}.\end{cases}
\end{equation}

For any positive integer $M$, we let $K_p^{(0)}(M)$ be the subgroup of $\GL_2(\Z_p)$
consisting of elements that are congruent to
$$
 \mat{\ast}{0}{\ast}{\ast} \pmod{M}.
$$
We define $\Cl_d(M)$ as follows:
$$
 \Cl_d(M) = T_d(\A) /T_d(\Q)T_d(\R)\Pi_{p<\infty} (T_d(\Q_p) \cap K_p^{(0)}(M)).
$$
where
\begin{equation}\label{localBessellemmaeq0b}
 T_d=\{g\in\GL_2:\:^tgS_dg=\det(g)S_d\}.
\end{equation}
It is easy to see that $\Cl_d(M)$ is finite. For example, $\Cl_d(1)$ is canonically isomorphic to the ideal class group of $\Q(\sqrt{d})$. Let $c \in \Cl_d(M)$ and let $t_c \in \prod_{p<\infty}T_d(\Q_p)$ be a representative for $c$. By strong approximation, we can write (non-uniquely)
\begin{equation}\label{strongapproxeq}
 t_{c} = \gamma_{c}\,\gamma_{c,\infty}^{-1}\,\kappa_{c}
\end{equation}
with $\gamma_{c} \in \GL(2,\Q)^+$, and $\kappa_{c}\in \Pi_{p<\infty} K_p^{(0)}(M)$.
\par
It is known (see~\cite[p.\ 209]{Fu}) that $S_c:=\det(\gamma_{c})^{-1}(^t{\gamma_{c}})S_d\gamma_{c} \in \mathcal{P}_2$ and satisfies $c(S_c) = 1$, $\disc(S_c)=d$. Also, the $(2,2)$-coefficient of $S_c$ is $1$ modulo $M$. For any positive integer $L$, we define
\begin{equation}\label{philmdef}
 \phi_{L,M}(c) =\mat{L}{}{}{L}\mat{M}{}{}{1}S_c\mat{M}{}{}{1}.
\end{equation}
It follows that $c(\phi_{L,M}(c))=L$, $\disc(\phi_{L,M}(c))=dL^2M^2$. We remark here that the matrix $\phi_{L,M}(c)$ depends on our choice of representative $t_c$ as well as on our choice of the matrix $\gamma_c$ involved in strong approximation. However, the equivalence class $[\phi_{L,M}(c)]$ is independent of these choices. In fact, we have the following proposition.

\begin{proposition}For each pair of positive integers $L,M$, the map $c \mapsto [\phi_{L,M}(c)]$ gives a well-defined bijection between $\Cl_d(M)$ and $H(dM^2; L)$.
\end{proposition}
\begin{proof} Let us first show that the map is well defined. Let $\Gamma^{0}(M)$ (resp. $\Gamma_{0}(M)$) be the usual congruence subgroups of $\SL_2(\Z)$ consisting of matrices whose upper-right (resp. lower-left) entry is divisible by $M$. Let $c \in \Cl_d(M)$. Suppose that $\gamma_c^{(1)}$, $\gamma_c^{(2)}$ are two distinct elements obtained in~\eqref{strongapproxeq} from $c$ and that $\phi_{L,M}^{(1)}(c)$, $\phi_{L,M}^{(2)}(c)$ are the matrices obtained via~\eqref{philmdef}. Then our definitions imply that there exists $t \in T(\Q)$, $k \in \Gamma^{0}(M)$ such that $\gamma_c^{(2)} = t \gamma_c^{(1)} k$. It follows immediately that $ \phi_{L,M}^{(2)}(c) = {}^tR\phi_{L,M}^{(1)}(c) R$
where
$$
 R=\mat{M^{-1}}{}{}{1}k\mat{M}{}{}{1}\in\Gamma_{0}(M)\subset\SL_2(\Z).
$$
Hence $[\phi_{L,M}^{(2)}(c)] = [\phi_{L,M}^{(1)}(c) ]$.

Next, we show that the map $c \mapsto [\phi_{L,M}(c)]$ is injective. Suppose that $[\phi_{L,M}(c_1)] = [\phi_{L,M}(c_2)]$. Then there exists $A\in \SL_2(\Z)$ such that ${}^tA \phi_{L,M}(c_2) A =  \phi_{L,M}(c_1)$. An easy calculation involving the entries of the matrices shows that $A \in \Gamma_0(M)$. Then $R =\mat{M}{}{}{1}A\mat{M^{-1}}{}{}{1}\in \Gamma^{0}(M)$ and ${}^tR S_{c_2} R =  S_{c_1}.$ Let $t_1 = \gamma_1\gamma_{1,\infty}^{-1}\kappa_{1}$ and $t_2 = \gamma_2\gamma_{2,\infty}^{-1}\kappa_{2}$ be our chosen representatives in $\prod_{p<\infty}T(\Q_p)$ of $c_1$ and $c_2$, respectively. Then
$$
 \gamma_2R\gamma_1^{-1} \in T_d(\Q) \cap t_2\GL_2(\R)^{+}\prod_{p}K_p^{(0)}(M)t_1^{-1}.
$$
It follows that  $t_1$ and $t_2$ represent the same element of $\Cl_d(M)$, completing the proof.

Finally, we show that the map $c \mapsto [\phi_{L,M}(c)]$ is surjective. Since we have already proved injectivity, it is enough to show that $|\Cl_d(M)| = |H(dM^2; L)|.$ Now it is a classical fact (see e.g.~\cite[p.\ 217]{Co}) that
$$
 |H(dM^2; L)| = |H(dM^2; 1)| = \frac{M}{u(d)}\,|\Cl_d(1)|\,\prod_{p|M} \left( 1 -  \Big(\frac{d}p\Big) p^{-1} \right),
$$
where $u(-3)=3$, $u(-4)=2$ and $u(d)=1$ for other $d$. On the other hand, a  simple argument along the lines of~\cite[p.\ 68]{S-th}, shows that
$$
 |\Cl_d(M)| = \frac{M}{u(d)}\,|\Cl_d(1)|\,\prod_{p|M} \left( 1 -  \Big(\frac{d}p\Big) p^{-1} \right).
$$
This completes the proof.
\end{proof}

We now return to the proof of Theorem~\ref{t:globalmain}. In order to prove the first assertion, it is enough to prove that (for some fixed fundamental discriminant $d<0$, fixed positive integers $L$, $M$ and fixed elements $c_1, c_2 \in \Cl_d(M)$) that \begin{equation}\label{fourierinveq1}
 a(\phi_{L,M}(c_1)) = a(\phi_{L,M}(c_2)).
\end{equation}
Let $t_1 \in \prod_{p<\infty}T_d(\Q_p)$ be a representative for $c_1$ and let $t_2 \in \prod_{p<\infty}T_d(\Q_p)$ be a representative for $c_2$. By Lemma~\ref{extra-inv-lem}, it follows that
\begin{equation}\label{fourierinveq2}
 \Phi_F^{S_d}(t_1g) = \Phi_F^{S_d}(t_2g) \qquad \text{ for all } g \in G(\A).
\end{equation}
Define the element $(H_{L,M})_f
\in \prod_{p<\infty}G(\Q_p)$ to be the diagonal embedding in $\prod_{p<\infty}G(\Q_p)$ of the element
$$
 H_{L,M} = \begin{bmatrix}LM^2&&&\\&LM&&\\&&1&\\&&&M\end{bmatrix}.
$$
Let us explicate~\eqref{fourierinveq2} in the special case $g=(H_{L,M})_f.$ Note that for $i=1,2$, we have
$$
 t_i (H_{L,M})_f = \gamma_{c_i} (H_{L,M})_\Q (H_{L,M})_\infty^{-1} (\gamma_{c_i})_\infty^{-1} (H_{L,M})_f^{-1}\kappa_{c_i}(H_{L,M})_f,
$$
where we note that $\gamma_{c_i} (H_{L,M})_\Q \in M(\Q)$ and $(H_{L,M})_\infty^{-1} (\gamma_{c_i})_\infty^{-1} \in M(\R)^+$. Furthermore, the element $(H_{L,M})_f^{-1}\kappa_{c_i}(H_{L,M})_f$ lies in the group $K_0$ defined after \eqref{spin strong approx}. It follows from Proposition~\ref{fourier-relation-prop} that
\begin{equation}\label{fourierinveq3}
 \begin{split}&\Phi_F^{S_d}(t_1(H_{L,M})_f) = e^{-2\pi\,{\rm Tr}(S_d)}(LM)^{-k} a(\phi_{L,M}(c_1)), \\ & \Phi_F^{S_d}(t_2(H_{L,M})_f) = e^{-2\pi\,{\rm Tr}(S_d)}(LM)^{-k} a(\phi_{L,M}(c_2)). \end{split}
\end{equation}
Combining~\eqref{fourierinveq2} and~\eqref{fourierinveq3}, we deduce~\eqref{fourierinveq1}. This completes the proof of the first assertion of Theorem~\ref{t:globalmain}. To prove the second assertion of Theorem~\ref{t:globalmain}, we need the following result, which is Theorem 2.10 of \cite{KST}.

\begin{theorem}[Kowalski--Saha--Tsimerman]\label{t:relationkst}
  Let $d<0$ be a fundamental discriminant and $\Lambda=\prod_p \Lambda_p$ be a character of $\Cl_d(1)$ (note that $\Lambda$ induces a character on $\Cl_d(M)$ for all positive integers $M$ via the natural surjection $\Cl_d(M) \rightarrow \Cl_d(1)$). Let $F$ be a cuspidal Siegel Hecke eigenform of weight $k$ with respect to $\Symp_4(\Z)$ and let $\pi=\otimes_p \pi_p$ be the irreducible cuspidal representation of $\GSp_4(\A)$ attached to $F$. For $S \in \mathcal P_2$, let $a(S)$ denote the Fourier coefficient of $F$ at $S$. For each prime $p$, let $B_p$ be the spherical vector in the $(S_d, \Lambda_p, \theta_p)$-Bessel model for $\pi_p$, normalized so that $B_p(1)=1$. Then for any positive integers $L = \prod_p p^{l_p}$ and $M= \prod_p p^{m_p}$ the relation
  $$
   \frac{1}{|\Cl_d(M)|}\sum_{c\in \Cl_d(M)}
   \overline{\Lambda(c)}a(\phi_{L,M}(c)) =
   \bigg(\frac{(LM)^{k}}{|\Cl_d(1)|}\sum_{c\in\Cl_d(1)} \overline{\Lambda(c)}a(S_c)\bigg)
   \prod_{p\mid LM}B_{p}(h(l_p, m_p))
  $$
  holds.
\end{theorem}

Let us see what this implies in the setup of  Theorem~\ref{t:globalmain}. For $F$  a cuspidal Siegel Hecke eigenform of weight $k$ with respect to $\Symp_4(\Z)$ which is a Saito-Kurokawa lift, $D<0$ a discriminant and $L$ a positive integer, define $a(D;L)=a(S)$ where $S$ is any member of $\mathcal P_2$ satisfying $c(S)=L$ and $\disc(S)=DL^2$ (this is well-defined by the first assertion of Theorem~\ref{t:globalmain}, which we have already proven).  Then Theorem~\ref{t:relationkst}, in the special case $\Lambda=1$, $F$ as above, tells us that

\begin{equation}\label{penul}
a(dM^2;L) =
(LM)^{k} a(d;1)
\prod_{p\mid LM}B_{p}(h(l_p, m_p)).
\end{equation}
We need to prove
\begin{equation}\label{maass-rels-eqn2p}
  a(dM^2;L) = \sum\limits_{r | L} r^{k-1} a(d(LM/r)^2;1).
\end{equation}
By~\eqref{penul}, the left side is equal to
$$
 (LM)^{k} a(d;1)\prod_{p\mid LM}B_{p}(h(l_p, m_p)),
$$
and the right side is equal to
$$
 \sum\limits_{r | L} r^{k-1} (LM/r)^k a(d;1) \prod_{p\mid LM}B_{p}(h(0, l_p+ m_p - r_p)).
$$
Hence, we are reduced to showing that
$$
 \prod_{p\mid L}B_{p}(h(l_p, m_p))=\sum\limits_{r | L} r^{-1}\prod_{p\mid L}B_{p}(h(0, l_p+ m_p - r_p)).
$$
This equation would follow provided for each prime $p \mid LM$ we could prove that
$$
 B_{p}(h(l_p, m_p))=\sum_{i=0}^{l_p}p^{-i} B_{p}(h(0, l_p+ m_p -i)).
$$
But this follows from Theorem~\ref{localmaasstheorem}. Note here that, by our remarks in Sect.~\ref{specialautformsec}, the non-archimedean local components $\pi_p$ associated to $F$ are of the form $\chi1_{\GL(2)}\rtimes\chi^{-1}$ for an unramified character $\chi$ of $\Q_p^\times$ (type IIb). This concludes the proof of Theorem \ref{t:globalmain}.\qed
\section{Normalization of the Bessel functions}
We return to the setup of Section~\ref{specialautformsec}. We will prove certain explicit formulas for the Bessel functions and their effect under change of models. This will lead to another proof of the classical Maass relations which does not use Theorem~\ref{t:relationkst}. In the following let
$$
 S = \mat{a}{b/2}{b/2}{c}\in \mathcal{P}_2
$$
be fixed. Recall that in Sect.~\ref{specialautformsec} we have fixed distinguished vectors $v_p^0$ in each local representation $V_p$. For any place $p$ let $\beta_p$ be a non-zero functional $V_p\to\C$ as in Lemma \ref{localBessellemma}. Let $B_p^S$ be the \emph{Bessel function} corresponding to $v_p^0$, i.e., $B_p^S(g)=\beta_p(\pi_p(g)v_p^0)$ for $g\in G(\Q_p)$. We are going to normalize the $\beta_p$, hence the $B_p^S$, in a certain way.
\subsubsection*{Non-archimedean case}
Assume that $p$ is a prime. It follows from Sugano's formula \eqref{suganosformulaeq} that $B_p^S(1)\neq0$, provided $S$ satisfies the standard assumptions \eqref{standardassumptionseq}.
%following conditions: $a, b \in \Z_p$ and $c \in \Z_p^\times$ and $\det(2S)$ is the generator of the discriminant of $\Q_p(\sqrt{-\det(2S)})/\Q_p$.
Hence, if $S$ satisfies these conditions, then we can normalize $B_p^S(1) = 1$.

For arbitrary positive definite $S$ in $\mathcal{P}_2$ we proceed as follows. Let $\disc(S) = N_1^2 d$, where $d$ is a fundamental discriminant. Let $S_d$ be as in~\eqref{defsd}. Then we have $S_d = a\,^t\!A S A$, where
\begin{equation}\label{A-defn}
 A := \begin{cases}
 \mat{-\frac b{2N_1a}}{\frac 1a}{\frac 1{N_1}}{} & \text{ if } d \equiv 0 \pmod{4},\\[4ex]
 \mat{-\frac b{2N_1a}+\frac 1{2a}\;}{\frac 1a}{\frac 1{N_1}}{} & \text{ if } d \equiv 1 \pmod{4}.
 \end{cases}
\end{equation}
For brevity, we put $S'=S_d$. We observe that the matrix $S'$ satisfies the standard assumptions \eqref{standardassumptionseq} for every prime $p < \infty$. Consequently, we can normalize our Bessel functions so that $B_p^{S'}(1) = 1$ for all primes $p$. However, for applications we require the values $B_p^S(1)=B_p^{S'}(\mat{A^{-1}}{}{}{a\,^t\!A})$. Calculating this value requires decomposing the argument of $B_p^{S'}$ in the form \eqref{doub-cos-eq} and then using Sugano's formula \eqref{suganosformulaeq}. We postpone the proof of the following lemma until the next section.

% In terms of $S$, this translates into
% \begin{equation}\label{nonarchnormeq}
%  B_p^{S}(\mat{A}{}{}{a^{-1}\,^t\!A^{-1}}) = 1.
% \end{equation}
\begin{lemma}\label{finding-lm}
 Let $A,S,S', d, N_1$ be as above. Let $L = c(S)$.
% It is not difficult to see that $\frac{N_1}{2^\epsilon L}$ is an integer.
   Then $\mat{A^{-1}}{}{}{a\,^t\!A}= t h(l,m) k$, with $t \in T_{S'}(\Q_p)$, $k \in \GSp_4(\Z_p)$ and
 $$
  m =v_p(N_1/L) \qquad \text{ and } \qquad  l = v_p(N_1) - m = v_p(L).
 $$
 Consequently, $B_p^S(1)=B_p^{S'}(h(l,m))$ with these values of $l$ and $m$. Since $S'$ satisfies the standard assumptions \eqref{standardassumptionseq}, the right hand side can be evaluated using Sugano's formula \eqref{suganosformulaeq}.
% In particular, if $L=1$, then $m = v_p\big(\frac{N_1}{2^\epsilon}\big)$ and $l=0$.
\end{lemma}
\subsubsection*{Archimedean case}
If $p = \infty$, let $A' = \mat{-b/(N_1a)}{1/a}{2/N_1}{0} \mat{-1/\sqrt{-d}}{}{}{1}$, where as before we write $\disc(S) = N_1^2 d$, with $d$ a fundamental discriminant. Then $S' = a\,^t\!A' S A' = 1_2$, the identity matrix. We normalize so that $B_\infty^{1_2}(1) = e^{-4\pi}$; this is possible by Theorem 3.10 of \cite{PS1}.
% In terms of $S$, this translates into
% \begin{equation}\label{archnormeq}
%  B_\infty^{S}(\mat{A'}{}{}{a^{-1}\,^t\!A'^{-1}}) = e^{-4\pi}.
% \end{equation}
What naturally appears when we relate Bessel models to Fourier coefficients is the value $B_\infty^{S}(1)=B_\infty^{1_2}(\mat{A'^{-1}}{}{}{a\,^t\!A'})$. Calculating this value requires decomposing the argument of $B_\infty^{1_2}$ as $thk$, where $t\in T_S(\R)$, the matrix $h$ is diagonal, and $k$ is in the standard maximal compact subgroup of $\SSp_4(\R)$. Then one may use the explicit formula given in Theorem 3.10 of \cite{PS1}. The result is as follows.
\begin{lemma}\label{findinglambdazeta}
 With the above notations and normalizations, we have
 \begin{equation}\label{B-infty-formula}
  B_\infty^{S}(1)= \det(S)^{k/2}\,e^{-2 \pi\,{\rm Tr}(S)}.
 \end{equation}
\end{lemma}
We postpone the proof of this lemma until the next section.
\subsection*{Global normalization}
Recall that our starting point was a Saito-Kurokawa lift $F\in S_k^{(2)}(\Gamma)$ and its associated adelic function $\Phi_F$ defined in \eqref{spin strong approx}. Having fixed the local vectors $v^0_p$ at each place, we may normalize the isomorphism $V\cong\otimes V_p$ such that $\Phi_F$ corresponds to the pure tensor $\otimes v^0_p$.

Given $S\in\mathcal{P}_2$, let $C_S$ be the constant defined by \eqref{extra-inv-lemeq1}. Having fixed the vectors $v_p^0$, the functionals $\beta_p$, and the isomorphism $V\cong\otimes V_p$, the constants $C_S$ are well-defined. By definition,
% This implies that, if $\Phi_F^{S}(1) \neq 0$ for some $S \in {\rm Sym}_2(\Q)$, then $\pi_F$ has a global $(\Theta_{S}, 1)$-Bessel model and it is given by $\Phi \mapsto \Phi^{S}$. Now, let $\pi_F = \otimes' \pi_p$, then each $\pi_p$ have local $(\Theta_{S,p}, 1)$-Bessel models. For $p < \infty$, the representations $\pi_p$ are of Type IIb (see Table A1 of \cite{RS}) and let $B_p^{S}$ be the spherical vector (unique up to scalars) in the $(\Theta_{S,p}, 1)$-Bessel model for $\pi_p$. Since, $F$ is holomorphic and scalar valued of weight $k$, the representation $\pi_\infty$ is the holomorphic discrete series representation $\mathcal{E}(k,k)$ (See \cite{PS1}). Let $B_\infty^{S}$ be the vector in the $(\Theta_{S,\infty}, 1)$-Bessel model of $\pi_\infty$ of scalar weight $k$. Then by uniqueness of local and global Bessel models, we have
\begin{equation}\label{local-global-reln}
 \Phi_F^{S}(g) = C_{S} \prod\limits_{p \leq \infty} B_p^{S}(g_p),
\end{equation}
for all $g=(g_p)\in G(\A)$. The values of the constants $C_S$ are unknown, but the following property will be sufficient to derive the Maass relations.
\begin{lemma}\label{const-lem}
 Let $S, S' \in \mathcal P_2$ such that $\Phi_F^{S}(1), \Phi_F^{S'}(1) \neq 0$ and $S' = v\,^t\!A SA$ for some $A \in \GL_2(\Q)$ and $v \in \Q^\times$. Then $C_{S} = C_{S'}$.
\end{lemma}
\begin{proof}
We have
$$
 \Phi_F^{S'}(1) = \Phi_F^{S}(\mat{A}{}{}{v^{-1}\,^t\!A^{-1}}) = C_{S} \prod\limits_{p \leq \infty} B_p^{S}(\mat{A}{}{}{v^{-1}\,^t\!A^{-1}}) = C_{S} \prod\limits_{p \leq \infty} B_p^{S'}(1).
$$
On the other hand, the left hand side equals $C_{S'} \prod\limits_{p \leq \infty} B_p^{S'}(1)$. The assertion follows.
\end{proof}

%\noindent
%\textcolor{blue}{{\bf Remark:} Assume that $F$ is the Saito-Kurokawa lift of $f\in S_{2k-2}(\SL_2(\Z))$ with $k$ even. By combining Proposition \ref{fourier-relation-prop}, equation \eqref{local-global-reln}, as well as Lemmas \ref{finding-lm} and \ref{findinglambdazeta} for the values of the local Bessel functions, one obtains a formula for the Fourier coefficient $a(S)$ in terms of the Satake parameters of $f$. This formula has been generalized to Hilbert-Siegel modular forms in \cite{O13}. The formula may be rewritten in terms of the Fourier coefficients of $f$, after which one recovers Lemma 3.1 of \cite{DKS}. See \cite{PS2} for details.}

%
%
%

\section{Explicit formulas for Bessel functions}
%\fbox{Omit this section completely for short version.}

In this section we provide the proofs of Lemmas \ref{finding-lm} and \ref{findinglambdazeta}.

% To obtain an explicit formula for $a(S)$, we will need the values of the local Bessel functions
% $$
%  B^S_p(1)=B_p^{S'}(\mat{A^{-1}}{}{}{a\,^t\!A})\quad(p<\infty)\qquad\text{and}\qquad B^S_\infty(1)=B_\infty^{1_2}(\mat{A'^{-1}}{}{}{a\,^t\!A'});
% $$
% see (\ref{Phi-A-reln}) and (\ref{local-global-reln}). Here, $A$ and $A'$ are the matrices defined in the previous section; they depend on $S$.
%
% We let $T_{S'}$ be the identity component of the stabilizer of $\Theta_{S'}$ in $M$. Explicitly, $T_{S'} = \{g \in \GL_2 : {}^tg S' g = \det(g) S'\}$; cf.\ \eqref{localBessellemmaeq0}. Let $R_{S'} = T_{S'} U$ be the corresponding \emph{Bessel subgroup}.
%
\subsection*{Non-archimedean case}
We first consider $p < \infty$.  Recall that $S'$ satisfies the standard assumptions \eqref{standardassumptionseq}.
%
% , then
% \begin{equation}\label{doub-cos-gl2}
%   \GL_2(\Q_p) = \bigsqcup_{m \geq 0} T_{S'}(\Q_p) \mat{p^m}{}{}{1} \GL_2(\Z_p);
% \end{equation}
% see Lemma 2-4 of \cite{Su}. In conjunction with the Iwasawa decomposition, this implies
% \begin{equation}\label{doub-cos}
%  \GSp_4(\Q_p) = \bigsqcup_{\substack{l, m \in \Z \\ m \geq 0}} R_{S'}(\Q_p) h(l,m) \GSp_4(\Z_p),
% \end{equation}
% where
% $$
%  h(l,m) = \begin{bmatrix}p^{l+2m}\\&p^{l+m}\\&&1\\&&&p^m\end{bmatrix}.
% $$
% Evidently, if $g = uth(l,m)k$ with $u \in U(\Q_p)$, $t \in T_{S'}(\Q_p)$ and $k \in \GSp_4(\Z_p)$, then $B_p^{S'}(g) = \Theta_{S'}(u) B_p^{S'}(h(l,m))$.
In view of \eqref{doub-cos-eq}, given an element $g \in \GSp_4(\Q_p)$ of the form $g = \mat{h}{}{}{v\,^th^{-1}}$, we want to know which double coset $R(\Q_p) h(l,m) \GSp_4(\Z_p)$ it belongs to. For this, we first state the following result for $\GL_2$.

\begin{lemma}\label{toricdecompositionlemma}
 Suppose that $S' = \mat{{\bf a}}{{\bf b}/2}{{\bf b}/2}{{\bf c}}$ satisfies the standard assumptions \eqref{standardassumptionseq}. Let $h\in\GL_2(\Q_p)$, and assume that, according to the decomposition \eqref{doub-cos-gl2-eq},
 \begin{equation}\label{toricdecompositionlemmaeq1}
  h=t\mat{p^m}{}{}{1}k
 \end{equation}
 with $t\in T_{S'}(\Q_p)$, a non-negative integer $m$, and $k\in\GL_2(\Z_p)$. Define $a',b',c',d'\in \Q_p$ by
 $$
  \mat{a'}{b'}{c'}{d'}=\det(h)^{-1}\,^thS'h.
 $$
 Then
 $$
  m = {\rm max}(0, -v(a'),-v(d')).
 $$
\end{lemma}
\begin{proof}
Taking determinants on the identity (\ref{toricdecompositionlemmaeq1}), we get
\begin{equation}\label{toricdecompositionlemmaeq2}
 \det(h)=\det(t)p^m\det(k).
\end{equation}
A calculation shows that
$$
 \det(h)^{-1}\,^thS'h = \det(k)^{-1}\,^tk\mat{\mathbf{a}p^m}{\mathbf{b}/2}{\mathbf{b}/2}{\mathbf{c}p^{-m}}k.
$$
%\begin{align*}
% \det(h)^{-1}\,^thS'h&=\det(h)^{-1}\,^tk\mat{p^m}{}{}{1}\,^ttS't\mat{p^m}{}{}{1}k\\
% &=\det(h)^{-1}\det(t)\,^tk\mat{p^m}{}{}{1}S'\mat{p^m}{}{}{1}k\\
% &=p^{-m}\det(k)^{-1}\,^tk\mat{p^m}{}{}{1}S'\mat{p^m}{}{}{1}k\\
% &=\det(k)^{-1}\,^tk\mat{1}{}{}{p^{-m}}S'\mat{p^m}{}{}{1}k\\
% &=\det(k)^{-1}\,^tk\mat{\mathbf{a}p^m}{\mathbf{b}/2}{\mathbf{b}/2}{\mathbf{c}p^{-m}}k.
%\end{align*}
If we let $k=\mat{w}{x}{y}{z}$, it follows that
% $$\det(h)^{-1}\,^thS'h= \det(k)^{-1}\mat{w^2\mathbf{a}p^m+wy\mathbf{b}+y^2\mathbf{c}p^{-m}}{wx\mathbf{a}p^m+(wz+xy)\frac{\mathbf{b}}2+yz\mathbf{c}p^{-m}}{wx\mathbf{a}p^m+(wz+xy)\frac{\mathbf{b}}2+yz\mathbf{c}p^{-m}}{x^2\mathbf{a}p^m+xz\mathbf{b}+z^2\mathbf{c}p^{-m}}.$$
%\begin{align*}
% &\det(h)^{-1}\,^thS'h=\det(k)^{-1}\mat{w}{y}{x}{z}\mat{\mathbf{a}p^m}{\mathbf{b}/2}{\mathbf{b}/2}{\mathbf{c}p^{-m}}\mat{w}{x}{y}{z}\\
% &\qquad=\det(k)^{-1}\mat{w\mathbf{a}p^m+y\mathbf{b}/2}{w\mathbf{b}/2+y\mathbf{c}p^{-m}}{x\mathbf{a}p^m+z\mathbf{b}/2}{x\mathbf{b}/2+z\mathbf{c}p^{-m}}\mat{w}{x}{y}{z}\\
% &\qquad=\det(k)^{-1}\mat{w^2\mathbf{a}p^m+wy\mathbf{b}+y^2\mathbf{c}p^{-m}}{wx\mathbf{a}p^m+(wz+xy)\mathbf{b}/2+yz\mathbf{c}p^{-m}}{wx\mathbf{a}p^m+(wz+xy)\mathbf{b}/2+yz\mathbf{c}p^{-m}}{x^2\mathbf{a}p^m+xz\mathbf{b}+z^2\mathbf{c}p^{-m}}.
%\end{align*}
$$
 a'=\det(k)^{-1}\big(w^2\mathbf{a}p^m+wy\mathbf{b}+y^2\mathbf{c}p^{-m}\big), \qquad
 d'=\det(k)^{-1}\big(x^2\mathbf{a}p^m+xz\mathbf{b}+z^2\mathbf{c}p^{-m}\big).
$$
Note that all quantities except (possibly) $p^{-m}$ are integral, and that $\mathbf{c}$ is a unit. Also, one of $y$ or $z$ is a unit. Hence,
$$
 a',d'\in\OF\quad\Longleftrightarrow\quad m=0.
$$
Assume $a'$ and $d'$ are not both in $\OF$. Then $m>0$, the valuation of $a'$ and $d'$ is $\geq-m$, and at least one of $a'$ or $d'$ has valuation $-m$. It follows that $-m={\rm min}(v(a'),v(d'))$, or equivalently $m={\rm max}(-v(a'),-v(d'))$. This concludes the proof.
\end{proof}

Using this lemma, it is straightforward to derive the following result for $\GSp_4$.
\begin{lemma}\label{whichlm}
 Let $g = \mat{h}{}{}{v\,^th^{-1}} \in \GSp_4(\Q_p)$ with $h \in \GL_2(\Q_p)$ and $v \in \Q_p^\times$. Then $g \in R_{S'}(\Q_p) h(l,m) \GSp_4(\Z_p)$, where $m$ is the integer obtained from Lemma \ref{toricdecompositionlemma} for $h$, and the value of $l$ is given by $l = v_p(\det(h)/v)-m$.
\end{lemma}
We can now give the \emph{proof of Lemma \ref{finding-lm}}: From Lemmas \ref{toricdecompositionlemma} and \ref{whichlm}, since
$$
 \det(A^{-1})^{-1} \,^t\!A^{-1} S' A^{-1} = aS\det(A)
$$
and $\det(A) = -1/ (N_1a)$, we get $m = {\rm max}(0, v_p\big(\frac{N_1}{a}\big), v_p\big(\frac{N_1}{ c}\big))$. Making use of ${\rm gcd}(a,b,c) = {\rm gcd}(a,c,N_1)$, the result follows.\qed
\subsection*{Archimedean case}
Now consider $p = \infty$. The values of $B_\infty^{S'}$ have been computed in \cite{PS1}. In this case $S' = 1_2$, and hence $T_{S'}(\R) = \{\mat{\gamma}{}{}{\gamma} : \gamma > 0\} \SO(2)$.  We have the following disjoint double coset decomposition,
$$
 \GSp_4(\R) = R(\R) \{h(\lambda, \zeta) : \lambda \in \R^\times,\:\zeta \geq 1 \} K^1,
$$
where $K^1$ is the maximal compact subgroup of $\SSp_4(\R)$, $R(\R) = T_{1_2}(\R) U(\R)$ is the Bessel subgroup, and
$$
 h(\lambda, \zeta) = \mat{\lambda \mat{\zeta}{}{}{\zeta^{-1}}}{}{}{\mat{\zeta^{-1}}{}{}{\zeta}}.
$$
If $g = uth(\lambda, \zeta)k_0$, with $u \in U(\R)$, $t \in T_{1_2}(\R)$ and $k_0 \in K^1$, then, by Theorem 3.10 of \cite{PS1},
\begin{equation}\label{B-infty-known-formula}
 B_\infty^{1_2}(g) = \Theta_{1_2}(u) B_\infty^{1_2}(h(\lambda, \zeta)) = \Theta_{1_2}(u)  \lambda^k e^{-2 \pi \lambda (\zeta^2+\zeta^{-2})}.
\end{equation}
Now we need to obtain the decomposition of $g_{A'} = \mat{A'^{-1}}{}{}{a\,^t\!A'}$ as $uth(\lambda, \zeta)k_0$. Clearly, $u=1$. Also,
$$
 A'^{-1} = \mat{}{-\frac{\sqrt{-d}N_1}2}{a}{\frac b2} = \mat{}{-1}{1}{} \mat{a}{\frac b2}{}{\frac{\sqrt{-d}N_1}2} = u \mat{}{-1}{1}{}  \mat{y}{x}{}{y^{-1}},
$$
with
$$
 u = \sqrt{a\frac{\sqrt{-d}N_1}2}, \qquad x = \frac{b/2}u, \qquad \text{ and } \qquad y = \frac au.
$$
\begin{lemma}\label{arch-toric-decomp}
 Let $h = \mat{y}{x}{}{y^{-1}}$ with $y \neq 0$. Then $h = k_1 \mat{\zeta}{}{}{\zeta^{-1}} k_2$, with $k_1, k_2 \in \SO(2)$ and
 $$
  \zeta^2 = \frac{1+x^2y^2+y^4 + \sqrt{(1+x^2y^2+y^4)^2-4y^4}}{2y^2}.
 $$
\end{lemma}
\begin{proof}
We may assume that $x\neq0$. By the Cartan decomposition of $\SL_2(\R)$, there exist $k_1, k_2 \in \SO(2)$ and $\zeta > 1$ such that $h = k_1 \mat{\zeta}{}{}{\zeta^{-1}} k_2$. Write $k_1 = \mat{\cos(\theta)}{\sin(\theta)}{-\sin(\theta)}{\cos(\theta)}$ for $\theta \in [0, 2\pi]$. Applying both sides of $h = k_1 \mat{\zeta}{}{}{\zeta^{-1}} k_2$ to $i$ as fractional linear transformations, and using that $\SO(2)$ stabilizes $i$, we get
$$
 y^2i+xy = \frac{\cos(\theta)\zeta^2 i + \sin(\theta)}{-\sin(\theta)\zeta^2 i + \cos(\theta)}.
$$
%Multiplying both sides by $-\sin(\theta)\zeta^2 i + \cos(\theta)$ and
Simplifying and comparing the coefficients of $i$ and the constant terms, we get
%$$
% -\zeta^2xy \sin(\theta) + y^2 \cos(\theta) = \zeta^2 \cos(\theta), \qquad \zeta^2y^2 \sin(\theta) + xy\cos(\theta) = \sin(\theta).
%$$
%Simplifying this, we get
$$
 -\zeta^2xy \sin(\theta)  = \cos(\theta) (\zeta^2-y^2), \qquad (1-\zeta^2y^2) \sin(\theta) = xy \cos(\theta).
$$
Note that, since $x, y \neq 0$, we have $\sin(\theta), \cos(\theta) \neq 0$ and $y \neq \pm\zeta, \pm 1/\zeta$. Hence, we can divide the above two equations and after simplification obtain $y^2 \zeta^4 - (1+x^2y^2+y^4) \zeta^2 + y^2 = 0$, which gives the lemma.
\end{proof}

% \begin{lemma}\label{findinglambdazeta2}
%  Let $A' = \mat{b/(\sqrt{-d}N_1a)}{1/a}{-2/(\sqrt{-d}N_1)}{0}$ and $g_{A'} = \mat{A'^{-1}}{}{}{a\,^t\!A'}$. Then $g_{A'} = th(\lambda, \zeta)k_0$, with $t \in T_{1_2}(\R)$, $k_0 \in K^1$ and
%  $$
%   \lambda = \frac{\sqrt{-d}N_1}2, \qquad \text{ and } \qquad \zeta = \sqrt{\frac{a+c+\sqrt{b^2+(a-c)^2}}{\sqrt{-d}N_1}}.
%  $$
%  In particular,
%  $$
%   \lambda (\zeta^2+\zeta^{-2}) = a+c = {\rm Tr}(S),
%  $$
%  and
%  \begin{equation}\label{B-infty-formula2}
%   B_\infty^{1_2}(\mat{A'^{-1}}{}{}{a\,^t\!A'}) = \det(S)^{k/2} e^{-2 \pi {\rm Tr}(S)}.
%  \end{equation}
% \end{lemma}

\vspace{2ex}
Finally, we can give the \emph{proof of Lemma \ref{findinglambdazeta}}: Apply Lemma \ref{arch-toric-decomp} with $x=(b/2)/u$ and $y=a/u$, where $u = \sqrt{a\frac{\sqrt{-d}N_1}2}$, to get
$$
 A'^{-1} = u k_1 \mat{\zeta}{}{}{\zeta^{-1}} k_2,
$$
with $\zeta$ as in the statement of the lemma and $k_1, k_2 \in \SO(2)$. Then
$$
 g_{A'} = au^{-1} \mat{k_1}{}{}{{}^tk_1^{-1}} \mat{\frac{u^2}a\mat{\zeta}{}{}{\zeta^{-1}}}{}{}{\mat{\zeta^{-1}}{}{}{\zeta}} \mat{k_2}{}{}{{}^tk_2^{-1}},
$$
as required. Now \eqref{B-infty-formula} follows from \eqref{B-infty-known-formula}.
\qed
\section{Formulas for Fourier coefficients}
%\fbox{Omit this section completely for short version.}

We again consider the situation where $F \in S_k^{(2)}(\SSp_4(\Z))$, $k$ even, is a Hecke eigenform and is a Saito-Kurokawa lift of $f \in S_{2k-2}(\SL_2(\Z))$. For every prime $p < \infty$ let $\alpha_p$ be the Satake parameter of $f$ at $p$. Let $a(S)$ denote the Fourier coefficient of $F$ at $S \in \mathcal P_2$.

\begin{proposition}\label{local-global-const-lem}
 Let $S = \mat{a}{b/2}{b/2}{c} \in \mathcal{P}_2$ be such that $a(S) \neq 0$. Assume that $\disc(S) = N_1^2d$ where $d<0$ is a fundamental discriminant. Let $L = c(S) = {\rm gcd}(a,b,c)$. For any prime $p < \infty$, set $a_p = v_p(N_1)$ and $b_p = v_p(L)$ and let $\delta_p = \big(\frac{d}p\big)$ be the Legendre symbol (by convention, $\delta_p=0$ if $p$ divides $d$). Let $C_S$ be the constant defined by \eqref{local-global-reln}. Then
 \begin{equation}\label{A(S)-explicit-eqn}
  a(S) =C_S(N_1)^{-3/2}\,\det(S)^{\frac k2} \prod\limits_{p | {N_1}} \sum\limits_{i=0}^{b_p} p^{\frac i2} \Big(\frac{\alpha_p^{a_p-i+1}-\alpha_p^{-a_p+i-1}}{\alpha_p-\alpha_p^{-1}} - \delta_p p^{-\frac 12} \frac{\alpha_p^{a_p-i}-\alpha_p^{-a_p+i}}{\alpha_p-\alpha_p^{-1}}\Big).
 \end{equation}
\end{proposition}
\begin{proof}
By Proposition \ref{fourier-relation-prop} and \eqref{local-global-reln},
$$
 a(S) e^{-2 \pi {\rm Tr}(S)} =  \Phi_F^{S}(1) = C_{S} \prod\limits_{p \leq \infty} B_p^{S}(1).
$$
Using Lemma \ref{findinglambdazeta}, it follows that
$$
 a(S)=C_{S}\,\det(S)^{k/2}\prod\limits_{p<\infty} B_p^{S}(1).
$$
Hence, by Lemma \ref{finding-lm},
\begin{align*}
 a(S)&=C_{S}\,\det(S)^{k/2}\prod\limits_{p|N_1}B_p^{S'}\Big(h\Big(v_p(L),v_p\Big(\frac{N_1}{L}\Big)\Big)\Big)\\
 &=C_{S}\,\det(S)^{k/2}\prod\limits_{p|N_1}B_p^{S'}(h(b_p,a_p-b_p)).
\end{align*}
By our remarks in Sect.~\ref{specialautformsec}, the non-archimedean local components $\pi_p$ of the automorphic representation $\pi_F$ are of the form $\chi_p1_{\GL(2)}\rtimes\chi_p^{-1}$ for an unramified character $\chi_p$ of $\Q_p^\times$ (type IIb). In fact, $\chi_p$ is the unramified character with $\chi_p(p)=\alpha_p$. Substituting the formulas in iii) and v) of Theorem \ref{localmaasstheorem} for the local Bessel functions, we obtain \eqref{A(S)-explicit-eqn}.
\end{proof}

Proposition \ref{local-global-const-lem} has been generalized to Hilbert-Siegel modular forms in \cite{O13}.

% This, in fact, gives a representation theoretic understanding of the Siegel series by obtaining a relation to spherical vectors in Bessel models. We record this in the following corollary.
% \begin{corollary}\label{Siegel-ser-cor}
% Let $\psi$ be an additive character of $\Q_p$ defined by $\psi(x) = exp(-2 \pi i y)$, where  $y \in \Z[1/p]$ such that $x-y \in \Z_p$. For $X \in {\rm Sym}_2(\Q_p)$ define $\nu(X) = [X \Z_p^2 + \Z_p^2 : \Z_p^2]$. For $S \in {\rm Sym}_2(\Q)$, semi-integral, and $s \in \C$, let the Siegel series of degree $2$ over $\Q_p$ be defined by
% $$b(S, s) = \sum\limits_{X \in {\rm Sym}_2(\Q_p)/{\rm Sym}_2(\Z_p)} \psi({\rm Tr}(SX)) \nu(X)^{-s}.$$
% Let $N, N_1, \epsilon, \alpha_p, L, \delta_p$ be as in the statement of Theorem \ref{local-global-const-lem}. Let $\pi_p$ be the representation of $\GSp_4(\Q_p)$ of Type IIb with parameter $\alpha_p$ as in \eqref{alpha-defn}. Let $\pi_p$ be given by its $(\Theta_{S',p},1)$-Bessel model and let $B_p^{S'}$ be the spherical vector in $\pi_p$ with $B_p^{S'}(1) = 1$. Let $p | N_1/2^\epsilon$ and let $s_0 \in \C$ be such that $p^{-s_0} = p^{-3/2} \alpha_p$. Then, we have
% \begin{equation}\label{Siegel-ser-eqn}
% b(S,s_0) = \frac{(1-p^{-3/2}\alpha_p)(1-p^{-1}\alpha_p^2)}{1-\delta_pp^{-1/2}\alpha_p} \alpha_p^{v_p(\frac{N_1}{2^\epsilon})}B_p^{S'}(h(v_p(L), v_p(\frac{N_1}{2^\epsilon}) - v_p(L))).
% \end{equation}
% \end{corollary}
% \begin{proof}
% This follows by comparing the explicit formula of $a(S)$ obtained in Theorem \ref{local-global-const-lem} above and Theorem \fbox{??} in \cite{O13}.
% \end{proof}

\subsection*{The formula of Das-Kohnen-Sengupta}
Let $f(\tau)=\sum_{n=1}^\infty c(n)e^{2\pi in\tau}\in S_{2k-2}(\SL_2(\Z))$ be an elliptic eigenform, with $k$ even, and assume that $F$ is the Saito-Kurokawa lift of $f$. If $\alpha_p$ is the Satake parameter at $p$, then
$$
 c(p^\mu)=p^{\mu(2k-3)/2}\,\frac{\alpha_p^{\mu+1}-\alpha_p^{-\mu-1}}{\alpha_p-\alpha_p^{-1}}.
$$
Hence formula \eqref{A(S)-explicit-eqn} may be rewritten as
\begin{align*}
  a(S)&=C_SN_1^{-3/2}\,\det(S)^{\frac k2} \prod\limits_{p | N_1} \sum\limits_{i=0}^{b_p} p^{\frac i2}p^{-(a_p-i)(2k-3)/2} \Big(c(p^{a_p-i})- \delta_p p^{-\frac 12} p^{(2k-3)/2}c(p^{a_p-i-1})\Big)\\
   &=C_SN_1^{-k}\det(S)^{\frac k2} \prod\limits_{p | N_1}\sum\limits_{i=0}^{b_p} p^{i(k-1)} \Big(c(p^{a_p-i})- \delta_p p^{k-2}c(p^{a_p-i-1})\Big).
\end{align*}
Now assume that $L=c(S) = N_1$. This is equivalent to saying that $S$ is a multiple of a matrix with fundamental discriminant. More precisely, putting $n=L=N_1$, we see that $S=nT$ where $T \in \mathcal{P}_2$ is such that $\disc(T) = d$. Hence,
\begin{align*}
  a(nT)&=C_{nT}\,n^{-k}\det(S)^{\frac k2} \prod\limits_{p|n}\,\sum\limits_{i=0}^{b_p} p^{i(k-1)} \Big(c(p^{a_p-i})- \delta_p p^{k-2}c(p^{a_p-i-1})\Big)\\
   &=C_T\,n^{-k}\det(S)^{\frac k2} \prod\limits_{p^\nu||n}\,\sum\limits_{i=0}^{\nu} p^{i(k-1)} \Big(c(p^{\nu-i})- \delta_p p^{k-2}c(p^{\nu-i-1})\Big)\\
   &=C_T\,\det(T)^{\frac k2} \prod\limits_{p^\nu||n}\,\sum\limits_{i=0}^{\nu} p^{(\nu-i)(k-1)} \Big(c(p^i)- \delta_p p^{k-2}c(p^{i-1})\Big).
\end{align*}
This coincides with the formula in Lemma 3.1 of \cite{DKS}. Comparison with this formula shows that $C_T\,\det(T)^{\frac k2}$ is a Fourier coefficient of the modular form of weight $k-1/2$ corresponding to $f$ under the Shimura lifting.
\section{Another proof of the classical Maass relations}
In this section we will give another proof of the Maass relations satisfied by the Fourier coefficients of a Saito-Kurokawa lift using our knowledge about Bessel models for the underlying automorphic representation, without recourse to the classical construction. This proof will not use Theorem~\ref{t:relationkst}.

\begin{theorem}\label{maass-rels-thm}
 Let $F$ be a cuspidal Siegel Hecke eigenform of weight $k$ with respect to $\Symp_4(\Z)$ which is a Saito-Kurokawa lift. For $S \in \mathcal P_2$, let $a(S)$ denote the Fourier coefficient of $F$ at $S$. Then
 \begin{equation}\label{maass-rels-eqn}
  a(\mat{a}{\frac b2}{\frac b2}{c}) = \sum\limits_{r | {\rm gcd}(a,b,c)} r^{k-1}\, a(\mat{\frac{ac}{r^2}}{\frac{b}{2r}}{\frac b{2r}}{1}).
 \end{equation}
\end{theorem}
\begin{proof}
Let $S = \mat{a}{\frac b2}{\frac b2}{c}$. As usual, we write $\disc(S) = N_1^2d$ where $d$ is a fundamental discriminant.
% Let us first look at the special case of ${\rm gcd}(a,b,c)=1$. Let $S_1 = \mat{ac}{\frac b2}{\frac b2}{1}$. Then, by Proposition \ref{local-global-const-lem} and Lemma \ref{const-lem}, we have
%\begin{equation}\label{d0=1}
%\frac{A(S)}{\det(S)^{\frac k2}\prod\limits_{p | \frac{N_1}{2^\epsilon}} B_p^{S'}(h(0,v_p\big(\frac{N_1}{2^\epsilon}\big)))} = C_S = C_{S_1} = \frac{A(S_1)}{\det(S_1)^{\frac k2}\prod\limits_{p | \frac{N_1}{2^\epsilon}} B_p^{S'}(h(0,v_p\big(\frac{N_1}{2^\epsilon}\big)))}.
%\end{equation}
%Since, $\det(S) = \det(S_1)$, we get the result $A(S) = A(S_1)$ in this special case.
Let $c(S)= L$. For $r | L$, set $S^{(r)} = \mat{\frac{ac}{r^2}}{\frac b{2r}}{\frac b{2r}}{1}$. By Proposition \ref{fourier-relation-prop} and \eqref{local-global-reln},
$$
 a(S) e^{-2 \pi {\rm Tr}(S)}=  \Phi_F^{S}(1) = C_{S} \prod\limits_{p \leq \infty} B_p^{S}(1).% = C_{S} B_\infty^{S}(1)\prod\limits_{p<\infty} B_p^{S'}(\mat{A^{-1}}{}{}{a\,^t\!A}).
$$
Using Lemma \ref{findinglambdazeta}, we get
$$
 a(S)=C_{S}\,\det(S)^{k/2}\prod\limits_{p<\infty} B_p^{S}(1).
$$
Hence, by Lemma \ref{finding-lm},
\begin{align*}
 a(S)&=C_{S}\,\det(S)^{k/2}\prod\limits_{p|L\;\text{or}\;p|(N_1/L)}
 B_p^{S'}\Big(h\Big(v_p(L),v_p\Big(\frac{N_1}{L}\Big)\Big)\Big)\\
  &=C_{S}\,\det(S)^{k/2}\bigg(\prod\limits_{p|L}B_p^{S'}
  \Big(h\Big(v_p(L),v_p\Big(\frac{N_1}{ L}\Big)\Big)\Big)\bigg)\bigg(\prod\limits_{\substack{p\nmid L\\p|N_1 }}B_p^{S'}\big(h(0,v_p(N_1))\big)\bigg).
\end{align*}
Analogously,
\begin{align*}
 a(S^{(r)})&=C_{S^{(r)}}\,\det(S^{(r)})^{k/2}
 \prod\limits_{p|N_1/r}B_p^{S'}\Big(h\Big(0,v_p\Big(\frac{N_1}
 {r}\Big)\Big)\Big)\\
  &=C_{S^{(r)}}\,\det(S^{(r)})^{k/2}\bigg(\prod\limits_{\substack{p|L\\p|N_1/r}}B_p^{S'}\Big(h\Big(0,v_p\Big(\frac{N_1}{r}\Big)\Big)\Big)\bigg)\bigg(\prod\limits_{\substack{p\nmid L\\p|N_1}}B_p^{S'}\big(h(0,v_p(N_1))\big)\bigg).
\end{align*}
If $N_1/r$ is not divisible by $p$, then $B_p^{S'}(h(0, v_p(N_1/r)))=1$ by our normalizations. Hence, the second condition under the first product sign can be omitted. Since $C_S=C_{S^{(r)}}$ by Lemma \ref{const-lem}, we conclude from the above equations that
$$
 r^{k-1}a(S^{(r)})\prod\limits_{p | L} B_p^{S'}\Big(h\Big(v_p(L), v_p\Big(\frac{N_1}{{} L}\Big)\Big)\Big) = \frac{a(S)}{r}\;\prod\limits_{p | L} B_p^{S'}\Big(h\Big(0, v_p\Big(\frac{N_1}{r}\Big)\Big)\Big).
$$
Applying $\sum_{r|L}$ to both sides gives
$$
 \Big(\sum_{r|L}r^{k-1}a(S^{(r)})\Big)\prod\limits_{p|L}B_p^{S'}\Big(h\Big(v_p(L), v_p\Big(\frac{N_1}{{} L}\Big)\Big)\Big)=a(S)\sum_{r|L}\frac{1}{r}\;\prod\limits_{p | L} B_p^{S'}\Big(h\Big(0, v_p\Big(\frac{N_1}{r}\Big)\Big)\Big).
$$
Since our assertion is $a(S)=\sum_{r|L}r^{k-1}a(S^{(r)})$, we are done if we can prove that
\begin{equation}\label{maass-rels-thmeq1}
 \prod\limits_{p|L}B_p^{S'}\Big(h\Big(v_p(L), v_p\Big(\frac{N_1}{{} L}\Big)\Big)\Big)=\sum_{r|L}\frac{1}{r}\;\prod\limits_{p | L} B_p^{S'}\Big(h\Big(0, v_p\Big(\frac{N_1}{r}\Big)\Big)\Big).
\end{equation}
Let $p_1^{\alpha_1}\cdot\ldots\cdot p_s^{\alpha_s}$ with $\alpha_j>0$ be the prime factorization of $L$. Then the right hand side of \eqref{maass-rels-thmeq1} equals
\begin{align*}
 &\sum_{i_1=0}^{\alpha_1}\ldots\sum_{i_s=0}^{\alpha_s}\,\frac{1}{p_1^{i_1}\cdot\ldots\cdot p_s^{i_s}}\;\prod\limits_{j=1}^s B_{p_j}^{S'}\Big(h\Big(0, v_{p_j}\Big(\frac{N_1}{{} p_1^{i_1}\cdot\ldots\cdot p_s^{i_s}}\Big)\Big)\Big)\\
 &\qquad=\sum_{i_1=0}^{\alpha_1}\ldots\sum_{i_s=0}^{\alpha_s}\;\prod\limits_{j=1}^s\Big(p_j^{-i_j} B_{p_j}^{S'}\Big(h(0, v_{p_j}(N_1)-i_j)\Big)\Big)\\
 &\qquad=\prod\limits_{j=1}^s\bigg(\sum_{i_j=0}^{\alpha_j}p_j^{-i_j} B_{p_j}^{S'}\Big(h(0, v_{p_j}(N_1)-i_j)\Big)\bigg).
\end{align*}
The left hand side of \eqref{maass-rels-thmeq1} equals
$$
 \prod\limits_{j=1}^sB_{p_j}^{S'}\Big(h(\alpha_j, v_{p_j}(N_1)-\alpha_j)\Big).
$$
Hence, \eqref{maass-rels-thmeq1} is proved if we can show that
\begin{equation}\label{maass-rels-thmeq2}
 B_{p_j}^{S'}\Big(h(\alpha_j, v_{p_j}(N_1)-\alpha_j)\Big)=\sum_{i=0}^{\alpha_j}p_j^{-i} B_{p_j}^{S'}\Big(h(0, v_{p_j}(N_1)-i)\Big)
\end{equation}
for all $j$. This follows from the implication ii) $\Rightarrow$ iii) of Theorem \ref{localmaasstheorem}, since, by our remarks in Sect.~\ref{specialautformsec}, the non-archimedean local components $\pi_p$ of $\pi_F$ are of the form $\chi1_{\GL(2)}\rtimes\chi^{-1}$ for an unramified character $\chi$ of $\Q_p^\times$ (type IIb).
%
%
%Since this is a purely local statement, let us write $B$ for $B_{p_j}^{S'}$, $\alpha$ for $\alpha_j$, $i$ for $i_j$, $\beta$ for $v_{p_j}(N_1/{})$, and $p$ for $p_j$. Then our assertion is
%\begin{equation}\label{maass-rels-thmeq3}
% B(h(\alpha,\beta-\alpha))=\sum_{i=0}^{\alpha}p^{-i} B(h(0,\beta-i))
%\end{equation}
%for integers $\alpha$, $\beta$ with $\beta\geq\alpha\geq0$. This identity does indeed hold, as we will prove in the following section.
\end{proof}

This theorem proves the Maass relations for the Saito-Kurokawa liftings constructed via representation theory, as indicated in the diagram \eqref{SKdiagrameq}. Since the proof comes down to properties of local $p$-adic Bessel functions, we expect similar relations to exist for other types of Siegel modular forms constructed from CAP type representations. This will be the topic of future investigations.
\section{A new characterization of Saito-Kurokawa lifts}
There are several different ways to characterize when a Siegel eigenform $F$ of weight $k$ with respect to $\Symp_4(\Z)$ is a Saito-Kurokawa lift. It was proved in the early 1980's by Maass, Andrianov and Zagier \cite{Ma,An2,Z} that $F$ is a Saito-Kurokawa lift if and only if $F$ satisfies the classical Maass relations~\eqref{maass-rels-eqn-3}. Since then, several other equivalent conditions have been discovered~\cite{H1,H2,PS3,RS3}; we refer to~\cite{PSFR} for a survey. All those conditions involve checking infinitely many relations on the Fourier coefficients of $F$. The advantage of those conditions, on the other hand, is that they also apply to non-eigenforms. Indeed, a Siegel cusp form $F$ of weight $k$ with respect to $\Symp_4(\Z)$ satisfies any of those conditions if and only if it lies in the Maass \emph{Spezialschar}, i.e., it is a linear combination of eigenforms that are Saito-Kurokawa lifts.

In~\cite{PSFR}, some new characterizations for a Siegel eigenform  being a Saito-Kurokawa lift were proved. These involved a single condition on the Hecke eigenvalues at a single prime. Unlike the conditions referred to in the previous paragraph, the new characterizations of~\cite{PSFR} are only applicable to eigenforms.

Below, we prove yet another condition that  if satisfied
implies that a Siegel eigenform is a Saito-Kurokawa lift. Like in~\cite{PSFR}, this new condition only applies to Hecke eigenforms (though we need this only at a single prime) and involves checking a single condition. However, unlike in~\cite{PSFR}, this new condition is phrased purely in terms of Fourier coefficients; thus, it is closer in spirit to the original relations of Maass, Andrianov, and Zagier. Indeed, this new condition, as phrased in~\eqref{avmaassrel} below, is nothing but a single ``Maass relation on average".

Let
$$
 F(Z) = \sum\limits_{S\in \mathcal{P}_2}a(S) e^{2 \pi i\,{\rm Tr}(SZ)}
$$
be a Siegel cusp form of weight $k$ with respect to $\Symp_4(\Z)$. Let $D<0$ be a discriminant and $L>0$ an integer. Recall that $H(D; L)$ denotes the set of equivalence classes of matrices in $\mathcal{P}_2$ whose content is equal to $L$ and whose discriminant is equal to $DL^2$. We define
\begin{equation}\label{atildedef}
 \widetilde{a}(D;L) = \frac{1}{|H(D; L)|}\sum_{[S]\in H(D; L)}a(S).
\end{equation}
In other words, $\widetilde{a}(D;L)$ is the \emph{average} of the Fourier coefficients of $F$ over matrices of discriminant $DL^2$ and content $L$. We now state our result.
\begin{theorem}\label{t:newsk}
 Let $F$ be a Siegel cusp form of weight $k$ with respect to $\Symp_4(\Z)$. Suppose that there is a prime $p$ such that $F$ is an eigenform for the local Hecke algebra at $p$ (equivalently, $F$ is an eigenform for the Hecke operators $T(p)$ and $T(p^2)$). For each discriminant $D<0$ and each positive integer $L$, let $\widetilde{a}(D;L)$ be defined as in~\eqref{atildedef}. Then the following are equivalent.
 \begin{enumerate}
  \item $F$ lies in the Maass \emph{Spezialschar}.
  \item There exists a fundamental discriminant $d<0$ such that $\widetilde{a}(d;1) \neq 0$ and
   \begin{equation}\label{avmaassrel}
    \widetilde{a}(d;p) = \widetilde{a}(dp^2;1) + p^{k-1}\,\widetilde{a}(d;1).
   \end{equation}
 \end{enumerate}
\end{theorem}
\begin{proof} Suppose that $F$ lies in the Maass \emph{Spezialschar}. Then  $\widetilde{a}(D;L)= a(D;L)$ where $a(D;L)$ is defined as in Theorem~\ref{t:globalmain}. It follows immediately from Theorem~\ref{t:globalmain} that~\eqref{avmaassrel} is satisfied for all $d$ and $p$. A special case of the main result of~\cite{S} tells us that there exists a fundamental discriminant $d<0$ such that   $\widetilde{a}(d;1) \neq 0$.

Next, suppose that there exists a fundamental discriminant $d<0$ such that   $\widetilde{a}(d;1) \neq 0$. Since $F$ is an eigenform of the local Hecke algebra at $p$, there exists a well-defined representation $\pi_p$ of $\GSp_4(\Q_p)$ attached to $F$. Indeed, for each  irreducible subrepresentation $\pi'$ of the representation of $G(\A)$ generated by $\Phi_F$, the local component of $\pi'$ at $p$ is isomorphic to $\pi_p$. By writing $F$ as a linear combination of Hecke eigenforms, applying Theorem~\ref{t:relationkst} on each of these eigenforms and then putting it all together, we conclude that, for all non-negative integers $l,m$,
$$
 \widetilde{a}(dp^{2m}; p^l) = p^{(l+m)k}\,\widetilde{a}(d; 1) B_{p}(h(l, m)),
$$
where $B_p$ is the spherical vector in the $(S_d, 1, \theta_p)$-Bessel model for $\pi_p$. The relation~\eqref{avmaassrel} together with the non-vanishing of $\widetilde{a}(d;1)$  now implies that \begin{equation}\label{final}B_{p}(h(1, 0))= B_{p}(h(0, 1)) + p^{-1}.\end{equation}
Using Theorem~\ref{localmaasstheorem}, it follows that one of the Satake parameters of $\pi_p$ is $p^{\pm1/2}$. Thus, each irreducible subrepresentation of the representation of $G(\A)$ generated by $\Phi_F$ is non-tempered at $p$. By a result of Weissauer~\cite{W}, it follows that each irreducible subrepresentation of the representation of $G(\A)$ generated by $\Phi_F$ is of CAP-type. This is equivalent to saying that $F$ lies in the Maass \emph{Spezialschar}.
\end{proof}

Theorem~\ref{t:newsk} tells us that if an eigenform $F$ lies in the orthogonal complement of the Maass \emph{Spezialschar} and satisfies $\widetilde{a}(d;1) \neq 0$ for some fundamental discriminant $d<0$, then
$$
 \widetilde{a}(d;p) \neq \widetilde{a}(dp^2;1) + p^{k-1}\,\widetilde{a}(d;1)
$$
for every prime $p$. We end this section with a slightly weaker version of this result that applies to non-eigenforms.

\begin{theorem}\label{t:newsk2}
 Let $F$ be a Siegel cusp form of weight $k$ with respect to $\Symp_4(\Z)$. Suppose that $F$ lies in the orthogonal complement of the Maass \emph{Spezialschar}. For each discriminant $D<0$ and each positive integer $L$, let $\widetilde{a}(D;L)$ be defined as in~\eqref{atildedef}. Let $d<0$ be a fundamental discriminant such that $\widetilde{a}(d;1) \neq 0$. Then, for all sufficiently large primes $p$, we have      \begin{equation}\label{avmaassrelnew1}
    \widetilde{a}(d;p) \neq \widetilde{a}(dp^2;1) + p^{k-1}\,\widetilde{a}(d;1).
   \end{equation}
\end{theorem}

\begin{proof}
Write $F = \sum_{i=1}^s F_i$ such that each $F_i$ is a Hecke eigenform. Let $\widetilde{a}_i(D;L)$ be the quantity corresponding to $F_i$. Denote $c_i = \widetilde{a}_i(d;1).$ So we have $\sum_{i=1}^s c_i = \widetilde{a}(d;1) \neq 0$.

Now, as in the proof of Theorem~\ref{t:newsk}, let $\pi_{i,p}$ be the representation of $\GSp_4(\Q_p)$ attached to $F_i$ at each prime $p$. Note that by our assumption, and by the result of Weissauer~\cite{W}, the representation $\pi_{i,p}$ is tempered. As before, we have $$
 \widetilde{a}_i(dp^{2m}; p^l) = p^{(l+m)k}\,\widetilde{a}_i(d; 1) B_{i,p}(h(l, m)),
$$
where $B_{i,p}$ is the spherical vector in the $(S_d, 1, \theta_p)$-Bessel model for $\pi_{i,p}$. This implies that
\begin{equation}\label{finalnew}
 \widetilde{a}(d;p) - \widetilde{a}(dp^2;1) - p^{k-1}\,\widetilde{a}(d;1) = p^k\sum_{i=1}^s c_i(B_{i,p}(h(1, 0))- B_{p}(h(0, 1)) - p^{-1}).
\end{equation}
Using Sugano's formula, and using the fact that the local parameters of $\pi_{i,p}$ lie on the unit circle, it is easy to check that
$$
 \lim_{p\rightarrow \infty}p(B_{i,p}(h(1, 0))- B_{i,p}(h(0, 1)) - p^{-1})=-1.
$$
So~\eqref{finalnew} implies that
$$
 \lim_{p\rightarrow \infty}p^{1-k} (\widetilde{a}(d;p) - \widetilde{a}(dp^2;1) - p^{k-1}\,\widetilde{a}(d;1)) = -\sum_{i=1}^s c_i = -\widetilde{a}(d;1) \neq 0.
$$
It follows that $\widetilde{a}(d;p) - \widetilde{a}(dp^2;1) - p^{k-1}\,\widetilde{a}(d;1) \neq 0$ for all sufficiently large primes $p$.
\end{proof}

\begin{comment}The relations of the form~\eqref{avmaassrelm} may be viewed as a sort of ``Maass relation on average". Theorem~\ref{t:newsk} asserts that if $F$ is a Hecke eigenform at some prime then a single such ``Maass relation on average" is all that is needed to conclude that $F$ lies in the \emph{Spezialschar}.

Using similar methods and some additional work, it should be possible to show that if a Siegel cusp form for $\Symp_4(\Z)$ (not necessarily an eigenform for any Hecke operator) satisfies the relations
\begin{equation}\label{avmaassrelmn}
 \widetilde{a}(d;1) \neq 0, \quad \quad \widetilde{a}(d;p^l)=\sum\limits_{i =1}^l p^{i(k-1)} \widetilde{a}(dp^{2(l-i)};1)
\end{equation}
for $l=1, 2, \ldots, \mathrm{dim}(S_k(\Symp_4(\Z))$, then it lies in the Maass \emph{Spezialschar}. We suppress further discussion of this point in the interest of brevity.
\end{comment}

\end{document}